\documentclass[letter, 10pt]{article}
\usepackage{authblk}
\usepackage[numbers]{natbib}
\usepackage{geometry}
\usepackage{amsfonts,amsmath,amssymb,amsopn, amsthm}
\usepackage{graphicx}
\usepackage{epstopdf}
\usepackage{algorithmic}
\usepackage{algorithm}
\usepackage{hyperref}
\usepackage{verbatim}
\usepackage{enumitem}
\usepackage{mathrsfs}

\newtheorem{theorem}{Theorem}
\numberwithin{theorem}{section}
\newtheorem{lemma}[theorem]{Lemma}
\newtheorem{proposition}[theorem]{Proposition}
\newtheorem{definition}[theorem]{Definition}
\newtheorem{assumption}[theorem]{Assumption}

\theoremstyle{remark}
\newtheorem{remark}[theorem]{Remark}

\newcommand{\R}{\mathbb{R}}

\newcommand{\grad}{\nabla \!\!\;}
\newcommand{\hess}{\nabla^2 \!\!\;}

\newcommand{\K}{k}
\newcommand{\primalL}{L}
\newcommand{\primalmu}{\mu}
\newcommand{\dualL}{L^*}
\newcommand{\dualmu}{\mu^*}
\newcommand{\xmin}{x_{\min}}
\newcommand{\dualvarx}{x^*}
\newcommand{\dualvary}{y^*}

\newcommand{\inner}[2]{\ensuremath{\left\langle#1,#2\right\rangle }}

\newcommand{\norm}[1]{\ensuremath{\left\lVert#1\right\rVert }}

\newcommand{\etal}{\emph{et al.}}

\renewcommand{\l}{\left}
\renewcommand{\r}{\right}

\newcommand{\ie}{i.e.,}
\newcommand{\eg}{e.g.,}

\DeclareMathOperator{\interior}{int}
\DeclareMathOperator{\dom}{dom}

\DeclareMathOperator{\rank}{rank}

\newcommand{\abs}[1]{\left\lvert#1\right\rvert}

\title{Dual Space Preconditioning for Gradient Descent}
\author[1,4,*]{Chris J. Maddison}
\author[2,*]{Daniel Paulin}
\author[3]{Yee Whye Teh}
\author[3]{Arnaud Doucet}
\affil[1]{University of Toronto, Toronto, Canada}
\affil[2]{University of Edinburgh, Edinburgh, UK}
\affil[3]{University of Oxford, Oxford, UK}
\affil[4]{DeepMind, London, UK}
\affil[*]{Both authors contributed equally to this work.}
\date{\today}

\begin{document}

\maketitle

\begin{abstract}
The conditions of relative smoothness and relative strong convexity were recently introduced for the analysis of Bregman gradient methods for convex optimization. We introduce a generalized left-preconditioning method for gradient descent, and show that its convergence on an essentially smooth convex objective function can be guaranteed via an application of relative smoothness in the dual space. Our relative smoothness assumption is between the designed preconditioner and the convex conjugate of the objective, and it generalizes the typical Lipschitz gradient assumption. Under dual relative strong convexity, we obtain linear convergence with a generalized condition number that is invariant under horizontal translations, distinguishing it from Bregman gradient methods. Thus, in principle our method is capable of improving the conditioning of gradient descent on problems with non-Lipschitz gradient or non-strongly convex structure. We demonstrate our method on $p$-norm regression and exponential penalty function minimization.
\end{abstract}
 \section{Introduction}

\subsection{Setting and method}
\label{sec:introduction}

We study the minimization of a proper, closed, and essentially smooth convex function $f: \R^d \to \R \cup \{\infty\}$,
\begin{equation}\tag{P}
\label{eq:problem} \min_{x \in \R^d} f(x).
\end{equation}
For unconstrained $f$, \ie{} $\dom f = \{x \in \R^d : f(x) < \infty\} = \R^d$, essential smoothness is simply differentiability. For constrained $f$, essential smoothness is the assumption that $f$ is differentiable on $\interior(\dom f) \neq \emptyset$ and that the norm of the gradient grows without bound, $\norm{\grad f(x)} \to \infty$, as $x$ approaches the boundary of the domain. Thus, a global minimizer $\xmin$ of $f$, if it exists, is in $\interior(\dom f)$. The method that we introduce (Algorithm \ref{alg:dualspaceprecon}) is a non-linear generalization of linear left-preconditioning for gradient descent (see, \eg{} \cite[sect. 9.4]{boyd2004convex}), and our analysis relies on recent generalizations of the typical Lipschitz gradient assumption \cite{bauschke2016descent}. For the sake of exposition, we will assume in the introduction that $f$ is twice continuously differentiable on $\interior(\dom f)$, but this is not a requirement of our method.

In the analysis of first-order methods, it is standard to assume that the derivatives of $f$ at some order are globally bounded by constants. For example, consider the gradient descent method, whose iterates satisfy
\begin{equation}
\label{eq:graddesc} x_{i+1} = \underset{x \in \dom f}{\arg \min} \left\{\inner{\grad f(x_i)}{x} + \tfrac{L}{2} \norm{x-x_i}^2\right\},
\end{equation}
where $L>0$ and $x_0 \in \interior(\dom f)$. A classical analysis shows that the iterates of gradient descent converge linearly in $i$, \ie{} $f(x_i) - f(\xmin) = \mathcal{O}(\lambda^{i})$ for $\lambda = 1 - \mu / L$, when $f$ is assumed to be $\mu > 0$ strongly convex and $\grad f$ is assumed to be $L$-Lipschitz continuous (typically called ``smoothness'').  Taken together for twice continuously differentiable $f$, these conditions are equivalent to the conditions that the eigenvalues of the Hessian matrix of second-order partial derivatives $\hess f(x)$ are everywhere lower bounded by $\mu > 0$ (strong convexity) and upper bounded by $L > 0$ (smoothness),
\begin{equation}
\label{eq:strgcvxsmoothtwicediff} \mu I \preceq \hess f(x) \preceq L I \text{ for all } x \in \interior(\dom f).
\end{equation}

\begin{algorithm}[t]
\caption{Dual preconditioned gradient descent.}
\label{alg:dualspaceprecon}
Given an essentially smooth convex $f : \R^d \to \R \cup \{\infty\}$, a Legendre convex $\K : \R^d \to \R \cup \{\infty\}$ with $\grad f(\interior( \dom f)) \subseteq \interior(\dom k)$ and $0 = \arg \min_{\dualvarx} k(\dualvarx)$, $x_0 \in \interior(\dom f)$, and $\dualL > 0$. For all $i \geq 0$,
\begin{equation*}
x_{i+1} = x_{i} - \frac{1}{\dualL} \grad \K(\grad f(x_{i})).
\end{equation*}
\end{algorithm}

Analyses of first-order methods using only non-constant bounds on the derivatives of $f$ have recently been discovered \cite{birnbaum2010new, bauschke2016descent, van2017forward, teboulle2018simplified, lu2018relatively}. In particular, \cite{bauschke2016descent} studied the following generalized gradient method that takes a designed essentially smooth, strictly convex reference function $h : \R^d \to \R \cup \{\infty\}$ with $\interior( \dom f) \subseteq \interior (\dom h)$. Given $x_0 \in \interior(\dom f)$, this method's iterates satisfy
\begin{equation}
\label{eq:mirrordescent}
x_{i+1} = \underset{x \in \dom f}{\arg \min} \left\{\inner{\grad f(x_i)}{x} + \primalL D_h(x, x_i)\right\}
\end{equation}
where $\primalL > 0$, $\inner{\cdot}{\cdot}$ is the Euclidean inner product, and $D_h(x,y) = h(x)-h(y) - \inner{\grad h(y)}{x-y}$ for $x,y \in \interior(\dom h)$. \eqref{eq:mirrordescent} is due to \cite{nemirovski1979effective} and falls in a family of so-called Bregman gradient methods. A standard analysis of \eqref{eq:mirrordescent} (see, \eg{} \cite{beck2003mirror}) makes the ``absolute'' assumptions that $f$ is Lipschitz continuous and that $h$ is strongly convex. In contrast, consider the following ``relative'' conditions between $f$ and $h$, for $\primalmu \geq 0$ and $\primalL > 0$
\begin{equation}
\label{eq:relcvxtwicediff} \primalmu \hess h(x) \preceq \hess f(x) \preceq \primalL \hess h(x)  \text{ for all } x \in \interior(\dom f).
\end{equation}
For twice continuously differentiable $f$, Bauschke \etal{} \cite{bauschke2016descent} first showed that \eqref{eq:relcvxtwicediff}  with $\primalmu = 0$ is a sufficient assumption to guarantee the sublinear convergence of $f(x_i) - f(\xmin)$ in \eqref{eq:mirrordescent}.
Lu \etal{} \cite{lu2018relatively} extended this analysis, and showed that \eqref{eq:relcvxtwicediff} with $\primalmu > 0$ is sufficient for the linear convergence of $f(x_i) - f(\xmin)$.  Conditions \eqref{eq:relcvxtwicediff} are ``relative'' in the sense that it is possible for \eqref{eq:relcvxtwicediff} to hold for $f$ and $h$ that are both non-smooth or non-strongly convex. For example, \cite{bauschke2016descent} study a Poisson inverse objective whose derivatives of all orders are unbounded as $x \to 0$. They design an appropriate $h$, whose Hessian is also unbounded at $0$, but which satisfies \eqref{eq:relcvxtwicediff}. Analyses of first-order methods using non-constant bounds on the derivatives of $f$ have been extended to non-convex $f$ \cite{bolte2018first, 2019arXiv190110791D} continuous convex optimization \cite{lu2019relative}, composite least-squares problems \cite{flammarion2017stochastic}, symmetric non-negative matrix factorization \cite{2019arXiv190110791D}, and the Sinkhorn algorithm \cite{2019arXiv190906918M}. Notably, relative smoothness conditions have also been used to justify fast implementations of third-order tensor methods \cite{nesterov2019implementable}.

The method that we introduce (Algorithm \ref{alg:dualspaceprecon}) exploits an application of these relative conditions in the dual space through an essentially smooth, strictly convex dual reference function $\K : \R^d \to \R \cup \{\infty\}$ with $\grad f(\interior(\dom f)) \subseteq \interior(\dom k)$ and $0 = \arg \min_{\dualvarx} k(\dualvarx)$. The method is a generalization of left-preconditioned gradient descent, which we discuss in more detail in section \ref{sec:preconditioning}. In section \ref{sec:dualgradient} we consider the conditions under which we can provide convergence rates for our method. For twice continuously differentiable $f$ sufficient conditions that we study are the existence of $\dualmu \geq 0$, $\dualL > 0$ such that
 \begin{equation}
     \label{eq:ourdualityconditions} \dualmu [\hess k(\grad f(x))]^{-1}\preceq \hess f(x) \preceq \dualL [\hess k(\grad f(x))]^{-1}  \text{ for all } x \in \interior(\dom f).
 \end{equation}
 When $\dualmu = 0$, we show that $\K(\grad f(x_i)) - k(0)$ converges sub-linearly with rate $\mathcal{O}(i^{-1})$ (and thus $x_i \to \xmin$) along the iterates of Algorithm \ref{alg:dualspaceprecon}. When $f$ is strictly convex and $\dualmu > 0$, we show that $f(x_i) - f(\xmin)$ converges linearly with rate $\lambda^* = 1-\dualmu/\dualL$. As we show in section \ref{sec:dualgradient}, assumptions \eqref{eq:ourdualityconditions} are relative smoothness and strong convexity assumptions in the dual space, and they are distinct from \eqref{eq:relcvxtwicediff}. In section \ref{sec:applications}, we design dual reference functions for $p$-norm regression (see \cite{bubeck2018homotopy,adil2019iterative} and references therein) and exponential penalty functions (see, \eg{} \cite{cominetti1994asymptotic, cominetti1994stable}).

\subsection{Preconditioning}
\label{sec:preconditioning}
In this paper, we introduce a generalization of linear left-preconditioning, which is a fundamental technique used in algorithms for solving linear systems. In this subsection, we review linear preconditioning, following closely Wathen's short introduction \cite{wathen_2015}, and give an interpretation of our method and Bregman gradient methods as left- and right-preconditioning, respectively.

Consider the problem of minimizing a positive-definite quadratic, which is equivalent to finding the solution $x$ of a linear system of $d$ equations with $d$ unknowns:
$Ax = b$ where $b \in \R^d$ and $A \in \R^{d \times d}$ is symmetric and positive-definite. ``Preconditioning'' refers to the idea of modifying this system in a way that preserves the solution, but improves the convergence of iterative methods. For example, given a positive-definite $P \in \R^{d \times d}$, we may consider the following systems (known as left- or right-preconditioning, respectively):
\begin{equation}
    \label{eq:linprecon}P^{-1}Ax = P^{-1}b \qquad \text{or} \qquad AP^{-1}y = b \, \text{ s.t. } x = P^{-1}y
\end{equation}
These have the same solution as the original, and if $P^{-1}A$ or $AP^{-1}$ approximates the identity, then iterative methods will converge faster. Indeed, for iterates of the conjugate gradients method (CG) \cite{hestenes1952methods}, $\inner{x-x_i}{A(x-x_i)}$ converges linearly with a rate that varies monotonically with the condition number $\kappa^A = \lambda_{\max}^A / \lambda_{\min}^A$, \ie{} the ratio of the largest to the small eigenvalue of $A$ \cite[Chap. 3.1]{greenbaum1997iterative}.
Smaller condition number is better, so if $\kappa^A \gg \kappa^{P^{-1}A}$, then left-preconditioned CG will converge faster. Preconditioned methods typical solve a system with $P$ at every iteration. Thus, $P$ should satisfy two criteria: $ \kappa^{P^{-1}A}$ should be small and $Px = b$ should be easy to solve. It may seem difficult to strike this balance, but it is possible in many cases. Wathen \cite{wathen_2015} gives an example due to Strang for Toeplitz matrices that reduces the complexity of linear solves from $\mathcal{O}(d^2)$ to $\mathcal{O}(d \log d)$. More generally, preconditioners are considered essential in solvers for very large, sparse, linear systems \cite{benzi2002preconditioning, saad2003iterative}.

Consider now the more general problem \eqref{eq:problem} for an unconstrained $f$. One can show that the following are stationary conditions of Algorithm \ref{alg:dualspaceprecon} or \eqref{eq:mirrordescent}, respectively:
\begin{equation}
    \label{eq:nonlinprecon} \grad \K(\grad f(x)) = 0 \qquad \text{or} \qquad \grad f(\grad h^*(y)) = 0 \, \text{ s.t. } x = \grad h^*(y),
\end{equation}
where $\grad h^*(y) = \arg \max_{x \in \R^d} \inner{x}{y} - h(x)$. Clearly, \eqref{eq:linprecon} specializes \eqref{eq:nonlinprecon} for appropriately chosen quadratic $f, k, h$. Thus, our method and the Bregman gradient method \eqref{eq:mirrordescent} may be seen as a generalization of left- and right-preconditioning for gradient descent, respectively. Moreover, for symmetric, positive-definite $A,P \in \R^{d \times d}$, the existence of $L,\mu > 0$ such that $\mu P \preceq A \preceq L P$ guarantees $\kappa^{P^{-1}A} \leq L / \mu$ and an error bound on preconditioned CG. This is generalized by the primal \eqref{eq:relcvxtwicediff} and dual  \eqref{eq:ourdualityconditions} relative conditions. However, in contrast to the linear case, the choice of left (dual) vs. right (primal) in the non-linear case is much more consequential and the two methods are not equivalent in general (left- and right-preconditioning for CG are equivalent \cite[Chap. 9.1]{saad2003iterative}). The class of $f$ satisfying the dual conditions \eqref{eq:ourdualityconditions} for a fixed $k$ is closed under horizontal translations. This is not true in general for $f$ satisfying the primal conditions \eqref{eq:relcvxtwicediff} for a fixed $h$. Thus, in general, $\primalmu \neq \dualmu$, $\primalL \neq \dualL$, and the global information encoded in the dual reference function $\K$ is distinct from the information encoded in the reference function $h$.

Non-linear preconditioning is far less studied, but has been considered in a number of works. Non-linear preconditioning methods have recently been shown to stabilize Euler discretization schemes of stochastic differential equations  \cite{hutzenthaler2012strong, sabanis2013note}. In fact, the non-linear preconditioning of \cite{hutzenthaler2012strong} is the same as the one we consider for exponential penalty functions. Finally, recent work \cite{cai2002nonlinearly, dolean2016nonlinear} developed non-linear preconditioning schemes for Newton's method applied to problems arising from the discretization of partial differential equations.

 \section{Convex analysis background}
\label{sec:cvxanalysis}
\newcommand{\sphere}{\mathcal{S}}

\subsection{Essential smoothness and convex conjugates}

In this section we review some basic facts of convex analysis that will be used throughout. Let $h : \R^d \to \R \cup \{\infty\}$ be a proper closed convex function with domain $\dom h = \{x : \R^d : h(x) < \infty\}$. To indicate $\dom h = \R^d$, we simply define $h : \R^d \to \R$ as ranging only over the reals. $\partial h(x)$ denotes the subdifferential of $h$ at $x \in \R^d$. For a proper convex functions, being closed is equivalent to being lower semi-continuous (lsc). Let $\norm{\cdot}$ and $\inner{\cdot}{\cdot}$ indicate the Euclidean norm and inner product, respectively, unless otherwise specified. The convex conjugate $h^* : \R^d \to \R \cup \{\infty\}$ of a proper closed convex function $h$ is given by
\begin{equation}
\label{eq:convconj}
h^*(\dualvarx)=\sup\{ \inner{x}{\dualvarx}-h(x) : x \in \dom h\}.
\end{equation}
$h^*$ is also a proper closed convex function, and $(h^*)^* = h$ \cite[Cor. 12.2.1]{rockafellar1970convex}. For more on $h^*$, we refer readers to \cite{rockafellar1970convex, boyd2004convex, borwein2010convex}.

In this work, we study the minimization of an essentially smooth convex function $f$ \cite[Chap. 25]{rockafellar1970convex}, which can be thought of as an assumption of differentiability. For constrained $f$, essential smoothness comes with additional structure that prevents $f$ from having sharp edges at the boundary of its domain. In some cases, we will consider the additional assumption that $f$ is strictly convex on the interior of its domain.

\begin{definition}[Essential smoothness and Legendre convexity]
\label{def:legendre}
Let $h : \R^d \to \R \cup \{\infty\}$ be a proper closed convex function. $h$ is essentially smooth if,
    \begin{enumerate}
        \item  \label{def:legendre:domain} $\interior(\dom h)$ is not empty.
        \item  \label{def:legendre:diff} $h$ is differentiable on $\interior(\dom h)$, with $\lim_{i \to \infty} \norm{\grad h(x_i)} = \infty$ whenever $x_i \in \interior(\dom h)$ is a sequence converging to the boundary of $\interior(\dom h)$.
    \end{enumerate}
$h$ is Legendre convex, if additionally
    \begin{enumerate}[resume]
        \item  \label{def:legendre:strictlyconvex} $h$ is strictly convex on $\interior(\dom h)$
    \end{enumerate}
In this work, the assumption that $h$ is essentially smooth carries with it the implied assumption that $h$ is proper and closed.
\end{definition}

Essentially smooth convex functions can only be minimized in their interior.

\begin{lemma}
\label{lemma:minimaoflegendrefns}
If $h : \R^d \to \R \cup \{\infty\}$ is an essentially smooth convex function that is minimized at $\xmin \in \dom h$, then $\xmin \in \interior(\dom h)$.
\end{lemma}
\begin{proof}
Suppose that $\xmin$ is a boundary point. Since $\interior(\dom h) \neq \emptyset$, by convexity there exists a line segment connecting the boundary point $\xmin$ and any other interior point $a$. However, by \cite[Lem. 26.2]{rockafellar1970convex}, we know that the directional derivative converges to $-\infty$ as we tend towards the boundary point on this line segment, hence $\xmin$ could not be a minimum of $h$.
\end{proof}

Legendre convex functions (essentially smooth, strictly convex functions) have even more convenient structure. One consequence of Legendre structure, which will be used in our analysis to show that $k$ is radially unbounded, is that achieving a minimum is sufficient to imply that a Legendre convex function grows without bound.

\begin{lemma}
\label{lemma:radiallyunbounded}
Let $h : \R^d \to \R \cup \{\infty\}$ be a Legendre convex function that is minimized at $0 \in \dom h$. Then $h$ is radially unbounded, \ie{} if $x_i \in \R^d$ is a sequence such that $\norm{x_i} \to \infty$, then $h(x_i) \to \infty$.
\end{lemma}
\begin{proof}
First, by Lemma \ref{lemma:minimaoflegendrefns} it follows that $0 \in \interior(\dom h)$. Because $h$ is strictly convex, $0$ is the unique minimum of $h$. Thus, we can define the sphere $\sphere = \{x \in \R^d : \norm{x} = r\}$ for some $r>0$ such that $\sphere \in \interior (\dom h)$. By continuity of $h$ in the interior of its domain, and the uniqueness of the minimum at zero, we have $\inf_{x\in \sphere} h(x)>h(0)$. Now, assume without loss of generality that $\norm{x_i} > r$. By strict convexity of Legendre functions, property \ref{def:legendre:strictlyconvex}, we have
\begin{equation}
    h(0) + \frac{\norm{x_i}}{r}\l(h\l(\frac{r x_i}{\norm{x_i}}\r) - h(0)\r) < h(0) + \l(h(x_i) - h(0)\r)
\end{equation}
and thus
\begin{equation}
    h(x_i)>h(0) + \frac{\norm{x_i}}{r} \l(\inf_{x\in \sphere} h(x) - h(0)\r).
\end{equation}
Our result follows by taking $i \to \infty$.
\end{proof}

 A second key consequence of Legendre structure is that the gradient map $\grad h$ is invertible and given by $(\grad h)^{-1} = \grad h^*$, which also gives a characterization of the inverse of $\hess h(x)$. We summarize both of these  properties in Lemma \ref{lemma:legendretwicediff}.

\begin{lemma}
\label{lemma:legendretwicediff}
Let $h : \R^d \to \R \cup \{\infty\}$ be Legendre convex. Then, $h^*$ is Legendre convex, the map $\grad h$ is one-to-one and onto from the open set $\interior(\dom h)$ onto the open set $\interior (\dom h^*)$, continuous in both directions, and for all $x \in \interior (\dom h)$
\begin{equation}
    \label{eq:conjugateinvgrad} \grad h^*(\grad h(x)) = x.
\end{equation}
If $h$ is $C^2$ on an open set containing $x$ and $\det \hess h(x) \neq 0$, then
\begin{equation}
    \label{eq:conjugateinvhess} \hess h^*(\grad h(x)) \hess h(x) =  \hess h(x) \hess h^*(\grad h(x)) = I.
\end{equation}
\end{lemma}
\begin{proof}
    For the first part see Rockafellar \cite[Thm. 26.5]{rockafellar1970convex}. For \eqref{eq:conjugateinvhess}, note that, by the inverse function theorem, $\grad h^*$ is continuously differentiable at $\grad h(x)$ under the assumption that $\grad h$ is continuously differentiable on an open set containing $x$. The remainder follows by the chain rule applied to (\ref{eq:conjugateinvgrad}).
\end{proof}

\subsection{Relative smoothness and relative strong convexity} Analyses of first-order methods for differentiable optimization typically require that $\grad f$ is Lipschitz continuous (smooth). Recent generalizations of smoothness (and strong convexity) \cite{bauschke2016descent, lu2018relatively} can be used to guarantee convergence of first-order methods beyond the Lipschitz $\grad f$ case. We will use these in our analysis of dual preconditioning. Following \cite{bauschke2016descent}, we define these relative conditions in terms of zeroth-order properties.
\begin{definition}[Relative smoothness and strong convexity]
  \label{def:relativesmoothness}
Let $h,g : \R^d \to \R \cup \{\infty\}$ be proper closed convex functions, $Q \subseteq \dom h \cap \dom g$ be a convex set, and $L, \mu \geq 0$. Define $d_L, d_{\mu} : \R^d \to \R \cup \{\infty\}$ for $x \in Q$ by
\begin{align}
    d_L(x) = L g(x) - h(x) \qquad d_{\mu}(x) = h(x) - \mu g(x)
\end{align}
and for $x \notin Q$ by $d_L(x) = d_{\mu}(x) = \infty$.
$h$ is $L$-smooth relative to $g$ on $Q$, if $d_L$ convex. $h$ is $\mu$-strongly convex relative to $g$ on $Q$, if $d_{\mu}$ is convex.
\end{definition}
\noindent The special cases with $g(x) = \norm{x}_2^2/2$ are exactly the classical conditions of strong convexity and smoothness.  We now provide first- and second-order characterizations.

\subsection{First-order characterizations for relative conditions}
 The first-order characterizations of relative smoothness and strong convexity are given in terms of the Bregman divergence \cite{bregman1967relaxation, banerjee2005clustering}, which for essentially smooth convex $h : \R^d \to \R \cup \{\infty\}$ and $x \in \dom h, y \in \interior (\dom h)$ is given by $h(x) - h(y) - \inner{\nabla h(y)}{x-y}$. Unfortunately, in our analysis, we will require smoothness relative to $f^*$, which can fail to be differentiable when $f$ is essentially smooth. Thus, we will make use of a generalization of the Bregman divergence, which we define via the \emph{one-sided directional derivative} of $h : \R^d \to \R \cup \{-\infty, \infty\}$ with respect to $y \in \R^d$ at a point $x$ where $h$ is finite:
  \begin{equation}
      h'(x; y) = \lim_{\epsilon \downarrow 0} \frac{h(x+ \epsilon y) - h(x)}{\epsilon}.
  \end{equation}
This may take values in $\{\infty, -\infty\}$. The advantage of one-sided direction derivatives is that they always exist at $x \in \dom h$ for proper convex $h$ \cite[Thm. 23.1]{rockafellar1970convex}. We are now prepared to define a novel generalization of the Bregman divergence.

\begin{definition}[Generalized Bregman divergences]
Let $h : \R^d \to \R \cup \{-\infty, \infty\}$ be a function. Let $x,y \in \R^d$ be points at which $h$ is finite and $h'(y; x-y)$ exists. Define the generalized Bregman divergence,
\begin{equation}
    D_h(x,y) = h(x) - h(y) - h'(y; x-y).
\end{equation}
If $h$ is proper, closed, and convex, then $D_h(x,y)$ is defined for all $x,y \in \dom h$ \cite[Thm. 23.1]{rockafellar1970convex} and is finite for $y\in \dom \partial h = \{x \in \R^d : \partial h(x) \neq \emptyset\}$ \cite[Thm. 23.2]{rockafellar1970convex}.
\end{definition}
\noindent Clearly, $D_h(x,y)$ coincides with the standard Bregman divergence if $h$ is differentiable at $y$. The advantage of the generalization is that it allows us to define the relative conditions in terms of first-order properties without the assumption of differentiability.

\begin{proposition}[First-order characterizations of relative conditions]
\label{prop:relativeconditions}
Let $h,g : \R^d \to \R \cup \{\infty\}$ be proper closed convex functions, $Q \subseteq \dom h \cap \dom g$ be a convex set, and $L, \mu \geq 0$.
The following are equivalent
\begin{enumerate}
    \item $h$ is $\primalL$-smooth relative to $g$ on $Q$.
    \item For all $x, y  \in Q$,
    $D_h(x,y) \leq \primalL D_g(x,y)$.
\end{enumerate}
The following are equivalent
\begin{enumerate}
    \item[3.] $h$ is $\primalmu$-strongly convex relative to $g$ on $Q$.
    \item[4.] For all $x, y  \in Q$,
    $\primalmu D_g(x,y) \leq D_h(x,y)$.
\end{enumerate}
\end{proposition}

To prove these equivalences, we will use two lemmas, which extend the first-order characterization of one-dimensional convexity to the non-differentiable case.

\begin{lemma}[A variant of the mean value theorem]
\label{lem:ourmvt}
 Let $r \colon [0,1] \to \R$ be a continuous function. Define $r'_+(z)= \lim_{\epsilon \downarrow 0} (r(z + \epsilon) - r(z))/\epsilon$. Assuming that $r'_+(z)$ exists for $z \in [0,1)$, if $s, t \in [0,1]$ and $s < t$, then there is exists $z \in [s,t)$ such that $r'_+(z) \ge (r(t) - r(s))/(t-s)$.
\end{lemma}
\begin{proof}
We can add a linear function to $r$ without changing the difference between the two sides on the inequality, so without loss of generality we may assume that $r(s) = r(t) = 0$.
Then, we want to prove that $r'_+(z) \ge 0$ for some $z \in [s, t]$.
Since $r$ is continuous, it has a minimum in $[s, t]$, so there is a $z \in [s,t]$ such that $r(u) \ge r(z)$ for every $u \in [s,t]$.
If $z = t$, then $r(z) = r(s)$, so we could instead take $z = s$.
Thus we may assume that $z \in [s,t)$.
Then $r'_+(z) = \lim_{u \downarrow z, \, u \in (z, t)} \frac{r(u)-r(z)}{u-z} \ge 0$, because $r(u) \ge r(z)$, proving the claim.
\end{proof}

\begin{lemma}[Characterization of one-dimensional convexity]\label{lem:convexfromaboveright}
Let $r \colon [0,1] \to \R$ be a continuous function. $r$ is convex on $[0,1]$ if and only if
$r'_+(z)$ (defined in Lemma \ref{lem:ourmvt}) exists for all $z \in (0,1)$, and for all $s,t \in (0,1)$ such that $s < t$,
\begin{equation}
    \label{eq:oneddirtangent} r(t) \ge r(s) + r'_+(s) (t-s).
\end{equation}
\end{lemma}
\begin{proof}
Suppose that $r$ is not convex on $[0,1]$. Then by continuity, it is also not convex on $(0,1)$.
After adding a linear function, we can arrange that $0 = r(s) = r(t) < r(z)$ for some $0 < s < z < t <1$.
Since $r$ is continuous, $r$ achieves its maximum restricted to the interval $[s,t]$, so we could choose $z \in (s,t)$ such that $r(z) = \max_{u \in [s,t]} r(u) > 0$.
Since $r(z)>0$ and $r$ is continuous, there is a $u \in (s,z)$ such that $r(v) > 0$ for every $v \in [u,z]$.
By Lemma \ref{lem:ourmvt}, there is a $v \in [u,z)$ such that $r_+'(v) \ge \frac{r(z) - r(u)}{z-u} \ge 0$, and so $0 = r(t) \ge r(v) + r'_+(v) (t-v) \ge r(v) > 0$.
This contradiction proves that $r$ is indeed convex.

Now suppose that $r$ is convex and continuous on $[0,1]$. By \cite[Thm. 23.1]{rockafellar1970convex} the difference quotient $\frac{r(z+\epsilon)-r(z)}{\epsilon}$ is a non-decreasing function of $\epsilon$ for $\epsilon > 0$ and $z \in [0, 1)$, and limit $r'_+(z)$ exists.
Using the fact that the difference quotient is non-decreasing in $\epsilon$ it follows that $r(t) \ge r(s) +  r'_+(s)(t-s)$ for every $s,t\in [0,1]$, $s<t$.
\end{proof}

We are now prepared to provide the proof of equivalence between the zeroth- and first-order definitions of the relative conditions.

\begin{proof}[Proof of Proposition \ref{prop:relativeconditions}]
We only show the equivalence of 1. and 2., the proof of the equivalence of 3. and 4. is similar. First, suppose that 1. holds, \ie{} $d_L$ is convex. Let $x,y \in Q$. Then for $x_t = y + t(x-y)$, we have (after dividing by $t$ and rearranging)
\begin{align}\label{eq:dLconvexityargument}
    h(x) - h(y) - \frac{h(x_t) - h(y)}{t} \leq  Lg(x) - Lg(y) - L\frac{g(x_t) - g(y)}{t}
\end{align}
Taking the limit $t \downarrow 0$ gives us that $D_h(x,y)\le L D_g(x,y)$, with the existence of the limits following from \cite[Thm. 23.1]{rockafellar1970convex}.

For the other direction, suppose that $D_h(x,y)\le L D_g(x,y)$ for every $x,y\in Q$.
Let $x,y \in Q$, and $x_t = y + t(x-y)$. Then for any $0<s<t<1$, it is easy to check that both
$h'(x_t; x_s - x_t)$ and $g'(x_t; x_s - x_t)$ are finite (if one of the directional derivatives is non-finite, then this would contradict the convexity or finiteness of these functions over $Q$). Thus $D_h(x_s,x_t)$ and $D_g(x_s,x_t)$ are finite and satisfy $D_h(x_s,x_t)\le L D_g(x_s,x_t)$ . Thus, it follows that for any $0<s<t<1$,
\begin{equation}\label{eq:DdLpos}
D_{d_L}(x_s,x_t)=L D_g(x_s,x_t)-D_h(x_s,x_t)\ge 0.
\end{equation}
Let $r(t)=d_L(x_t)$ for $t\in [0,1]$ be the restriction of $d_L$ on the line segment between $x$ and $y$, then \eqref{eq:DdLpos} implies that the condition \eqref{eq:oneddirtangent} holds. $r$ is a continuous function by \cite[Thm. 10.2]{rockafellar1970convex}. Thus, by Lemma \ref{lem:convexfromaboveright}, $r$ is a convex function on $[0,1]$. This holds for all $x,y\in Q$, thus $d_L$ is convex.

\end{proof}

\subsection{Second-order characterizations of relative conditions}
Verifying relative smoothness or strong convexity is typically done via second-order conditions. Just as Lipschitz continuity of $\grad h$ can be characterized by a bound on $\hess h$, the relative conditions can be characterized by the second derivatives of $h$ and $g$ \cite{bauschke2016descent, lu2018relatively}. Proposition \ref{lemma:twicediffrelativecondition} allows $\hess h, \hess g$ to be undefined at a point, a slight generalization of the standard result that is useful in our analysis when $\hess f$ is undefined at $\xmin$.

\begin{proposition}[Second-order characterizations of relative conditions]
\label{lemma:twicediffrelativecondition}
Let  $h, g: \R^d \to  \R \cup \{\infty\}$ be proper closed convex functions that are differentiable on the interior of their domains. Let $Q \subseteq \interior(\dom g) \cap \interior(\dom h)$ be an open convex set,  $z \in Q$, and $L, \mu \geq 0$. If $h, g$ are $C^2$ on $Q \setminus \{z\}$, then
\begin{enumerate}
    \item $h$ is $\primalL$-smooth relative to $g$ on $Q$ if and only if,
\[\hess h(x) \preceq \primalL \hess g(x) \qquad \forall x \in Q \setminus \{z\}.\]
    \item $h$ is $\primalmu$-strongly convex relative to $g$ on $Q$  if and only if,
\[\primalmu \hess g(x) \preceq \hess h(x) \qquad \forall x \in Q \setminus \{z\}.\]
\end{enumerate}
\end{proposition}
\begin{proof}
Again, we prove the relative smoothness equivalence, and relative strong convexity follows similarly. For relative smoothness, $(\Rightarrow)$ follows from part one of \cite[Thm. 2.1.4]{nesterov2018lectures} applied to $d_L$ at $x \in Q$. For $(\Leftarrow)$, it is sufficient to prove convexity of the restriction of $d_L$ to an open line segment with endpoints $x,y \in Q$. Let $x_t = y + t(x-y)$ and $r(t) = d_L(x_t)$ for $t \in (0,1)$. Let $a \in (0, 1)$ be such that $x_a = z$, if it exists, or some arbitrary $a \in (0,1)$, otherwise. $d_L$ is continuously differentiable at all $x \in Q$ by \cite[Thm 25.5]{rockafellar1970convex}. Thus $r'(t) = \inner{\grad d_L(x_t)}{x-y}$ is a continuous and finite function of $t \in (0,1)$. If $r'$ is non-decreasing, then the argument of \cite[Thm. 4.4]{rockafellar1970convex} gives us our result. Thus, with a slight abuse of notation,
\begin{align*}
    r'(t) = r'(a) + r'(t) - r'(a) = r'(a) + \lim_{s \to a} \int_{s}^t \inner{x-y}{\hess d_L(x_t) (x-y)}.
\end{align*}
The limit is actually a one-sided limit, depending on $t \leq a$ or $t > a$. Either way, $\hess d_L(x_t) = L \hess g(x_t) - \hess h(x_t)$ is positive semi-definite, so $r'$ is non-decreasing.
\end{proof}

 \section{Analysis of the dual preconditioned scheme}
\label{sec:dualgradient}

\subsection{Motivation and assumptions}

Relative smoothness of $f$ with respect to a reference function $h$ is the key assumption under which \cite{bauschke2016descent, teboulle2018simplified, lu2018relatively} analyzed the convergence of Bregman gradient methods. We now build towards an analysis of dual space preconditioned gradient descent method (Algorithm \eqref{alg:dualspaceprecon}) using the assumption that $k$ is smooth relative to $f^*$. As shorthand to distinguish these two assumptions, we use the terms \emph{primal relative smoothness} to refer to the condition that $f$ is $L$-smooth relative to $h$ and \emph{dual relative smoothness} to refer to the condition that $k$ is $\dualL$-smooth relative to $f^*$. To motivate our assumption, consider the following idealizations.

Consider the Bregman gradient method update \eqref{eq:mirrordescent}, which can be rewritten as,
\begin{equation}
    \label{eq:mirrordescent2}
    x_{i+1} = \arg \min_{x \in \dom f} \left\{\inner{\grad f(x_i) - L\grad h(x_i)}{x} + L h(x)\right\}.
\end{equation}
In this form, it is clear that, if $h = f$ and $L = 1$, then the iteration would converge in a single step to the minimizer of $h = f$. This is an idealization, because a single iteration would be as expensive to compute as the original problem. The spirit behind primal relative smoothness is that the condition $h = f$ can be relaxed to admit $h$ for which the update \eqref{eq:mirrordescent2} is efficiently solvable and the iterates still converge.

Now, consider the case that $f$ is Legendre convex with a minimum at $\xmin$, and let $f_c^*(\dualvarx) = f^*(\dualvarx) - \inner{\dualvarx}{\xmin}$ for $\dualvarx \in \R^d$. Notice that $\grad f_c^*(\grad f(x)) = x - \xmin$ by Lemma \ref{lemma:legendretwicediff} and that Algorithm \ref{alg:dualspaceprecon} with $k = f_c^*$ and $L_i = 1$ would converge in a single step to the minimizer $\xmin$ of $f$. Thus, in analogy to the relative smoothness analysis of \cite{bauschke2016descent} in the primal space, the spirit behind our analysis under dual relative smoothness is that the requirement $k = f_c^*$ can be relaxed while maintaining the convergence of Algorithm \ref{alg:dualspaceprecon}. In particular, sufficient assumptions on $k$ are that it is minimized at $0$ and smooth relative to $f^*$.

 More precisely, our analysis of Algorithm \ref{alg:dualspaceprecon} uses following assumptions.
\begin{assumption}
\label{assumption:descent}
\begin{enumerate}
    \item $f : \R^d \to \R \cup \{\infty\}$ is convex and essentially smooth.
    \item $k : \R^d \to \R \cup \{\infty\}$ is Legendre convex and uniquely minimized at $0$.
    \item \label{assumption:keyassumption} $\grad f(\interior(\dom f)) \subseteq \interior(\dom k)$ and for all $x,y \in \interior(\dom f)$,
\begin{equation*}
    D_k(\grad f(y), \grad f(x)) \leq \dualL D_f(x,y)
\end{equation*}
\end{enumerate}
\end{assumption}
 As we show in the following sections, Assumption \ref{assumption:descent}.\ref{assumption:keyassumption} is a necessary condition of the relative smoothness of $k$ with respect to $f^*$.  Assumption \ref{assumption:descent}.\ref{assumption:keyassumption} is the assumption that requires the most effort to verify, since the convexity of $\dualL f^* - k$ will typically be difficult to check. For this reason, we also provide second-order sufficient conditions expressed (mostly) in terms of conditions on $f$ and $k$.

\subsection{Dual relative conditions for Legendre convex objectives}  When $f$ is essentially smooth and strictly convex (Legendre), we are able to provide clean characterizations of the dual relative conditions. In particular, Assumption \ref{assumption:descent}.\ref{assumption:keyassumption} is necessary and sufficient for the smoothness of $k$ relative to $f^*$ on $\interior(\dom f^*)$. We begin by linking $D_f$ and $D_{f^*}$ in what is a well-known identity for Legendre convex $f$.

\begin{lemma}
\label{lemma:dualdivergence}
If $f: \R^d \to  \R \cup \{\infty\}$ is an essentially smooth convex function, then
\begin{equation}
D_{f^*}(\grad f(y), \grad f(x)) \leq D_f(x,y)
\end{equation}
for all $x,y \in \interior{(\dom f)}$. If $f$ is Legendre convex, then this is an equality.
\end{lemma}
\begin{proof}
Note, by \cite[Cor. 26.4.1]{rockafellar1970convex}, we have $\grad f(\interior(\dom f)) = \dom \partial f^*$. Thus, $D_{f^*}(\grad f(y), \grad f(x))$ is finite for any $x,y \in \interior(\dom f)$. Note $x \in \partial f^*(\grad f(x))$. Now,
\begin{equation*}
    \begin{aligned}
        D_{f^*}(\grad f(y), \grad f(x)) &= f^*(\grad f(y)) - f^*(\grad f(x)) - (f^*)'(\grad f(x); \grad f(y) - \grad f(x))\\
        &\overset{(a)}{\leq}  f^*(\grad f(y)) - f^*(\grad f(x)) - \inner{x}{\grad f(y) - \grad f(x)}\label{eq:fstarBregmaninequality}\\
        &\overset{(b)}{=} - f(y) + \inner{\grad f(y)}{y} + f(x) - \inner{\grad f(x)}{x}  - \inner{x}{\grad f(y) - \grad f(x)}\\
        &= f(x) - f(y) + \inner{\grad f(y)}{y-x}=D_f(x,y)\\
    \end{aligned}
\end{equation*}
where $(a)$ follows from \cite[Thm. 23.2]{rockafellar1970convex}, $(b)$ follows from \cite[Thm. 26.4]{rockafellar1970convex}. If $f$ is Legendre convex, then by Lemma \ref{lemma:legendretwicediff}, $f^*$ is Legendre, $f^*$ is differentiable on $\interior(\dom f^*) = \grad f(\interior(\dom f))$, and $(a)$ is an equality \cite[Thm. 23.4]{rockafellar1970convex} .
\end{proof}

We can now provide first-order characterizations of the dual relative conditions.

\begin{proposition}[First-order characterization of dual relative conditions, Legendre convex case]
\label{lem:equivalentcondLegendre}
Let $f, k : \R^d \to \{\R, \infty\}$ be Legendre convex functions. The following are equivalent.
\begin{enumerate}
    \item \label{lem:equivalentcond:relsmooth}
    $k$ is $\dualL$-smooth relative to $f^*$ on $\interior(\dom f^*)$.
\item \label{lem:equivalentcond:bregman} $\grad f(\interior (\dom f))\subseteq \interior (\dom k)$, and for all $x,y \in \interior (\dom f)$,
    \begin{equation*}
        D_k(\grad f(y), \grad f(x)) \leq \dualL D_f(x,y).
    \end{equation*}
\end{enumerate}
The following are equivalent.
\begin{enumerate}
    \item[3.] $k$ is $\dualmu$-strongly convex relative to $f^*$ on $\interior(\dom f^*)$.
    \item[4.] $\grad f(\interior (\dom f))\subseteq \interior (\dom k)$, and for all $x,y \in \interior (\dom f)$,
    \begin{equation*}
        \dualmu D_f(x,y) \leq D_k(\grad f(y), \grad f(x)).
    \end{equation*}
\end{enumerate}
\end{proposition}
\begin{proof}
We prove the relative smoothness results, and the relative strong convexity ones follow similarly. First, notice that that $\grad f(\interior(\dom f)) = \interior(\dom f^*)$.
by Lemma \ref{lemma:legendretwicediff}. For (\ref{lem:equivalentcond:relsmooth} $\Rightarrow$ \ref{lem:equivalentcond:bregman}), by the definition of relative smoothness, we have $\interior(\dom f^*)\subseteq \dom k$, and since this is an open set, we necessarily have $\interior(\dom f^*)\in \interior(\dom k)$. By Proposition \ref{prop:relativeconditions} we have that for all $\dualvarx, \dualvary \in \interior (\dom f^*)$,
\begin{align}
    D_{k}(\dualvary, \dualvarx) \leq \dualL D_{f^*}(\dualvary, \dualvarx).
\end{align}
By Lemmas \ref{lemma:legendretwicediff} and \ref{lemma:dualdivergence}, this implies
\begin{align}
    D_{k}(\grad f(y), \grad f(x)) \leq \dualL D_{f^*}(\grad f(y), \grad f(x)) = \dualL D_{f}(x, y),
\end{align}
 for all $x, y \in \interior(\dom f)$. For (\ref{lem:equivalentcond:bregman} $\Rightarrow$ \ref{lem:equivalentcond:relsmooth}), by Lemma \ref{lemma:dualdivergence}, we have for all $x, y \in \interior(\dom f)$,
\begin{equation}\label{eq:LegendreDfDfstar}D_f(x,y)=D_{f^*}(\grad f(y), \grad f(x)).\end{equation}
Using this, Proposition \ref{prop:relativeconditions} implies that $k$ is $L^*$-smooth relative to $f^*$ on $\interior (\dom f^*)$.
\end{proof}

If $f$ is Legendre convex, then the dual relative conditions have a natural second-order characterization, which reveals the structure of the difference between them and primal relative conditions. Again, typically it is easiest to prove dual relative smoothness (or strong convexity) via these second-order conditions.

\begin{proposition}[Second-order characterizations of dual relative conditions, Legendre convex case]
\label{lemma:secondorderdualrelativecond}
Let $f : \R^d \to \R \cup \{\infty\}$ be Legendre convex, minimized at $\xmin$, and $C^2$ on $\interior (\dom f) \setminus \{\xmin\}$ such that $\det \hess f(x) \neq 0$ at $x \in \interior (\dom f) \setminus \{\xmin\}$. Let $k : \R^d \to \R \cup \{\infty\}$ be Legendre convex, $C^2$ on $\interior{(\dom f^*)} \setminus \{0\}$ such that $\det \hess k(\dualvarx) \neq 0$ at $\dualvarx \in \interior (\dom f^*) \setminus \{0\}$. Let $L, \mu \geq 0$.
\begin{enumerate}
    \item $k$ is $\dualL$-smooth relative to $f^*$ on $\interior(\dom f^*)$ if and only if,
        \[\hess f(x) \preceq \dualL [\hess k(\grad f(x))]^{-1} \qquad \forall x \in \interior(\dom f) \setminus \{\xmin\}.\]
    \item $k$ is $\dualmu$-strongly convex relative to $f^*$ on $\interior(\dom f^*)$ if and only if,
        \[\dualmu [\hess k(\grad f(x))]^{-1} \preceq \hess f(x) \qquad \forall x \in \interior(\dom f) \setminus \{\xmin\}.\]
\end{enumerate}
\end{proposition}

\begin{remark}
\label{remark:diffconditions}
It is well-known that the primal and dual relative conditions are equivalent in the case of $\hess h(x) = I = \hess k(\dualvarx)$ (see, \eg{} \cite{zalinescu2002convex, shalev2007online, kakade2009duality}). In particular, if $f$ is $\mu$-strongly convex and $L$-smooth on $\interior(\dom f)$, then its convex conjugate $f^*$ is $(1/L)$-strongly convex and $(1/\mu)$-smooth on $\interior(\dom f^*)$. In fact, for twice continuously differentiable $f$, the equivalence is a simple consequence of Propositions \ref{lemma:twicediffrelativecondition} and \ref{lemma:secondorderdualrelativecond}. However, this equivalence is not true in general.

Given a Legendre convex $g : \R \to \R \cup \{\infty\}$ define the following sets of functions
\begin{align}
    \mathscr{F}_g &= \{\text{Legendre convex } f : f \text{ is smooth and strongly convex relative to } g\},\\
    \mathscr{F}_g^* &= \{\text{Legendre convex } f : g \text{ is smooth and strongly convex relative to } f^*\}.
\end{align}
Let $k(\dualvarx) = |\dualvarx|^q/q$ for $\dualvarx \in \R$ and $1 < q < 2$. A simple argument by contradiction shows that $\mathscr{F}_k^* \nsubseteq \mathscr{F}_h$ for all twice continuously differentiable $h : \R \to \R$, implying that the primal and dual relative conditions are not equivalent in general. Consider
\begin{equation}
f_b(x)=|x-b|^p/p,
\end{equation}
for $p=\frac{q}{q-1}$ and $x \in \R$. First $f_b \in \mathcal{F}_k^*$ for all $b$, which follows from $[k''(f'_b(x))]^{-1} = (p-1)|x-b|^{p-2} = f_b''(x)$ and Proposition \ref{lemma:secondorderdualrelativecond}. On the other hand, suppose there is some twice continuously differentiable $h : \R \to \R$ such that $f_b \in \mathscr{F}_h$ for all $b$. Then there exists $\mu > 0$ such that $\mu h''(b) \leq f_b''(b) = 0$ for all $b$. This implies that $h''(x) \equiv 0$ and thus $h(x) \equiv 0$. However, this leads to a contradiction, because smoothness is violated: $f_b''(b+\epsilon) > 0 = L h''(x)$ for any $L, \epsilon > 0$.
\end{remark}

\begin{proof}[Proof of Proposition \ref{lemma:secondorderdualrelativecond}]
We prove the relative smoothness result, and the relative strong convexity one follows similarly. By Lemma \ref{lemma:legendretwicediff}, if $\grad f$ is continuously differentiable for $x \in \interior(\dom f) \setminus \{\xmin\}$, then $\grad f^*$ is continuously differentiable for $\dualvarx \in \interior(\dom f^*) \setminus \{0\}$ by the inverse function theorem. Thus, by Proposition \ref{lemma:twicediffrelativecondition} dual relative smoothness is equivalent to: for all $\dualvarx \in \interior(\dom f^*) \setminus \{0\}$,
\begin{align}
    \label{intermediate_relsmoothdual}\hess k(\dualvarx) \preceq \dualL \hess f^*(\dualvarx).
\end{align}
By Lemma \ref{lemma:legendretwicediff}, \eqref{intermediate_relsmoothdual} is equivalent to for all $x \in \interior(\dom f) \setminus \{\xmin\}$,
\begin{align}
    \hess k(\grad f(x)) \preceq \dualL [\hess f(x)]^{-1}.
\end{align}
Since $A^{-1} \preceq B^{-1}$ is equivalent to $B \preceq A$ for positive definite matrices, we are done.
\end{proof}

A major difference between the primal and dual relative conditions is the fact that dual relative conditions are invariant under horizontal translations of $f$. To see why, let $k$ be $\dualL$-smooth relative to $f^*$ on a convex set $Q$. Define $g(x) = f(x  - z)$ for $z \in \R^d$. Then, by \cite[Thm. 12.3]{rockafellar1970convex}, $g^*(\dualvarx) = f^*(\dualvarx) + \inner{z}{\dualvarx}$. Bregman divergences of functions that differ only in affine terms are identical (see \cite{banerjee2005clustering} for the differentiable case), so we have for all $\dualvarx, \dualvary \in Q$, $D_k(\dualvarx, \dualvary) \leq L^* D_{f^*}(\dualvarx, \dualvary) = L^* D_{g^*}(\dualvarx, \dualvary)$.
Thus $k$ is $L^*$-smooth relative to $g^*$ on $Q$. Invariance under horizontal translation is clearly easy to violate in the case of primal relative smoothness.

Even if $h$ is allowed to translate with $f$, the primal and dual relative conditions can lead to distinct conditioning. Given a positive definite $A \succ 0$, let
\begin{equation}
\begin{aligned}
    f(x) = \norm{Ax - b}^p/p, \qquad h(x) = \norm{x - A^{-1}b}^p/p, \qquad k(\dualvarx) &= \norm{\dualvarx}^q/q,
\end{aligned}
\end{equation}
for $1/p + 1/q = 1$ and $p > 2$. It can be shown that $f$ satisfies both the dual (with respect to $k$) and primal (with respect to $h$) relative conditions. Nonetheless, the condition numbers are distinct. A simple calculation reveals that in this case
\begin{equation}
\begin{aligned}
     \frac{L}{\mu} &= p^2 \l(\frac{\sigma_{\max}(A)}{\sigma_{\min}(A)}\r)^{p} \qquad \text{vs.} \qquad \frac{L^*}{\mu^*} &= (p-1)^{2}\l(\frac{\sigma_{\max}(A)}{\sigma_{\min}(A)}\r)^{4-q},
\end{aligned}
\end{equation}
where $\sigma_{\min}$ and $\sigma_{\max}$ are the smallest and largest singular values of $A$, respectively. Thus, the primal condition number is larger than the dual number (since $4-q=3-(p-1)^{-1} < p$ when $p > 2$). Similarly, the example $f(x)=\|Ax-b\|_{4}^4/4+\|Cx-d\|_2^2/2$ of \cite[p. 339]{lu2018relatively} can be shown to have better conditioning under the dual preconditioned method than under the Bregman gradient method.

\subsection{Dual relative conditions for essentially smooth objectives}
We now show that the smoothness of $k$ relative to $f^*$ on $\dom f^*$ is a sufficient condition for Assumption \ref{assumption:descent}.\ref{assumption:keyassumption}. We also provide a sufficient, second-order condition.
\begin{proposition}
\label{lem:equivalentcond}
Let $f, k: \R^d \to \R \cup \{\infty\}$ be essentially smooth convex functions. If $k$ is $\dualL$-smooth relative to $f^*$ on $\dom f^*$, then  $\grad f(\interior( \dom f)) \subseteq \interior(\dom k)$ and for all $x,y \in \interior (\dom f)$,
    \begin{equation}
        D_k(\grad f(y), \grad f(x)) \leq \dualL D_f(x,y).
    \end{equation}
\end{proposition}
\begin{proof}
Note that by \cite[Cor. 26.4.1]{rockafellar1970convex}, $\grad f(\interior(\dom f))  = \dom \partial f^* \subseteq \dom f^*$. We have $\dom f^*\subseteq \dom k$ by the definition of relative smoothness. By Proposition \ref{prop:relativeconditions} and Lemma \ref{lemma:dualdivergence}, for every $x,y\in \interior(\dom f)$ we have
\begin{equation}
    \begin{aligned}
        D_{k}(\grad f(y), \grad f(x)) \leq \dualL D_{f^*}(\grad f(y),\grad f(x))\leq \dualL D_f(x,y)
    \end{aligned}
\end{equation}
Now, we are going to show that $\grad f(\interior (\dom f)) \subseteq \interior(\dom k)$.
We argue by contradiction. Suppose there is a $\dualvarx \in \grad f(\interior (\dom f))$ such that $\dualvarx \notin \interior (\dom k)$.
So $\dualvarx$ has to be in $\dom k \setminus \interior(\dom k)$, \ie{} on the boundary of $\interior(\dom k)$. Using the essential smoothness of $k$, by \cite[Lemma 26.2]{rockafellar1970convex} it follows that for any $\dualvary \in \interior (\dom k)$, we have
\begin{equation}\label{eq:kdirderlimit}
k'(\dualvarx+\lambda(\dualvary-\dualvarx);\dualvary-\dualvarx)\downarrow -\infty\text{ as } \lambda\downarrow 0.
\end{equation}
We fix an arbitrary $\dualvary\in \interior (\dom k)$, and define the function $h:[0,1]\to \R$ as
$h(\lambda)=L^* f^*(\dualvarx+\lambda(\dualvary-\dualvarx))-k(\dualvarx+\lambda(\dualvary-\dualvarx))$. Then relative smoothness implies that $h$ is a finite, continuous convex function on $[0,1]$. However, for such a function we must have $\limsup_{\lambda\downarrow 0}h_+'(\lambda)<\infty$, since otherwise it could not be finite on $[0,1]$ by Lemma \ref{lem:convexfromaboveright}. By combining this with \eqref{eq:kdirderlimit}, it follows that
\[{f^*}'(\dualvarx+\lambda(\dualvary-\dualvarx);\dualvary-\dualvarx)\to \infty\text{ as } \lambda\downarrow 0.\]
Let $r:[0,1]\to \R$ be $r(\lambda)=f^*(\dualvarx+\lambda(\dualvary-\dualvarx))$, then this implies that $r'_+(\lambda)\to \infty$ as $\lambda\downarrow 0$, and by \eqref{eq:oneddirtangent} of Lemma \ref{lem:convexfromaboveright},
this contradicts the assumption that $f^*$ is finite in $\dom f^*$. Hence we must have $x^*\in \interior(\dom k)$, and $\grad f(\interior (\dom f)) \subseteq \interior(\dom k)$.
\end{proof}
The next proposition gives a second-order sufficient condition for Assumption \ref{assumption:descent}.\ref{assumption:keyassumption}.
\begin{proposition}
\label{lemma:secondorderdualrelativecondnonLegendre}
Let $f : \R^d \to \R \cup \{\infty\}$ be essentially smooth, and $C^2$ on $\interior (\dom f)$. Let $k : \R^d \to \R \cup \{\infty\}$ be Legendre convex, and $C^2$ on $\interior(\dom k)$. If $\grad f(\interior (\dom f))\subseteq \interior(\dom k)$,
$\det \grad^2 k(\dualvarx) \neq 0$ for all $\dualvarx\in \interior(\dom k)$, and
\begin{equation}
    \hess f(x) \preceq \dualL [\hess k(\grad f(x))]^{-1} \qquad \forall x \in \interior(\dom f),
\end{equation}
then $D_k(\grad f(y), \grad f(x)) \le L^* D_f(x,y)$ for every $x,y\in \interior(\dom f)$.
\end{proposition}
\begin{proof}
Let $x,y\in  \interior(\dom f)$, and let  $W \subseteq \R^d$ be a bounded open neighborhood of the segment $I=[\grad f(x), \grad f(y)]=\{t\grad f(x)+(1-t) \grad f(y): 0\le t\le 1\}$ such that $\mathrm{cl}(W) \subseteq  \interior(\dom k)$.
Let $\delta > 0$.
Let $\epsilon > 0$, and let $f_{\epsilon} : \R^d \to \R \cup \{\infty\}$, $f_{\epsilon}(z) = f(z) + \epsilon \sqrt{1+\|z\|^2}$.
Then $f_{\epsilon}$ is a Legendre convex function, and $\dom(f_{\epsilon}) = \dom(f)$.
We have $\grad f_{\epsilon}(z) = \grad f(z) + \epsilon (1+\|z\|^2)^{-\frac{1}{2}} z$ and $\grad^2 f_{\epsilon}(z) = \grad^2 f(z) + \epsilon (1+\|z\|^2)^{-\frac{3}{2}} ((1+\|z\|^2) I_d - z^T z) \succ 0$ for every $z \in \interior(\dom f)$.
So $\|\grad f_{\epsilon} (z) - \grad f(z)\| \le \epsilon$ and $0 \preceq \grad^2 f_{\epsilon}(z) - \grad^2 f(z) \preceq \epsilon I_d$ for every $z \in \interior(\dom f)$.
Let $I_{\epsilon}$ denote the segment $[\grad f_{\epsilon}(x), \grad f_{\epsilon}(y)] \subseteq \R^d$.
Choose $\epsilon$ small enough so that
\begin{enumerate}
\item the $2\epsilon$-neighborhood of $I_{\epsilon}$ is in $W$ (so $\mathrm{dist}(I_{\epsilon}, \R^d \setminus W) \ge 2\epsilon$).
\item  $\epsilon I_d \preceq \frac{\delta}{2} [\grad^2 k(w)]^{-1}$ for every $w \in W$.
\item  $\forall w_1, w_2 \in W$ s.t.\, $\|w_1-w_2\| \le \epsilon$, $[\grad^2 k(w_1)]^{-1} \preceq (1+\frac{\delta}{2 L^*}) [\grad^2 k(w_2)]^{-1}$  (uniform continuity of $(\grad^2 k)^{-1}$ on compact set $\mathrm{cl}(W)$ by Heine-Cantor theorem).
\end{enumerate}
We will show that $(L^*+\delta) f_{\epsilon}^* - k$ is convex when restricted to the segment $I_{\epsilon}$. Let $w \in I_{\epsilon}$ and $z = (\grad f_{\epsilon})^{-1}(w) \in \interior(\dom f)$.
Then $\grad f_{\epsilon}(z) = w$, and since $\|\grad f(z) - w\| \le \epsilon$, we get $\grad f(z) \in W$.
We have $\grad^2((L^*+\delta) f_{\epsilon}^* - k)(w) = (L^*+\delta) [\grad^2 f_{\epsilon}(z)]^{-1} - \grad^2 k(\grad f_{\epsilon}(z))$, and we would like to show that this is $\succeq 0$.
So we want to show $\grad^2 f_{\epsilon}(z) \preceq (L^*+\delta) [\grad^2 k(\grad f_{\epsilon}(z))]^{-1}$.
This follows from $\grad^2 f_{\epsilon}(z) \preceq \grad^2 f(z) + \epsilon I_d$ and $\epsilon I_d \preceq \frac{\delta}{2} [\grad^2 k(\grad f_{\epsilon}(z))]^{-1}$ and $\grad^2 f(z) \preceq L^* [\grad^2 k(\grad f(z))]^{-1} \preceq (L^* + \frac{\delta}{2}) [\grad^2 k(\grad f_{\epsilon}(z))]^{-1}$.
So for small enough $\epsilon$'s $(L^*+\delta) f_{\epsilon}^* - k$ is indeed convex when restricted to $I_{\epsilon}$.
Then $D_k(\grad f_{\epsilon}(y), \grad f_{\epsilon}(x)) \le (L^*+
\delta)D_{f^*_{\epsilon}}(\grad f_{\epsilon}(y),\grad f_{\epsilon}(x)) =(L^*+
\delta)D_{f_{\epsilon}}(x,y)$, using the convexity of $(L^*+\delta) f_{\epsilon}^* - k$ on $I_{\epsilon}$ combined with the same limiting argument as in \eqref{eq:dLconvexityargument}, and Lemma \ref{lemma:dualdivergence}. Taking $\epsilon \downarrow 0$, and then $\delta\downarrow 0$ we get $D_k(\grad f(y), \grad f(x)) \le L^* D_f(x,y)$.
\end{proof}

\subsection{Convergence rates for dual space preconditioned gradient descent}

In this section we show that Assumption \ref{assumption:descent} is sufficient to provide convergence rates for Algorithm \ref{alg:dualspaceprecon} on essentially smooth convex $f$. We find that $k(\grad f(x_i))) - k(0)$ converges with rate $\mathcal{O}(i^{-1})$. Under an additional dual relative strong convexity condition, we find that $f(x_i) - f(\xmin)$ converges with rate $\mathcal{O}((1 - \dualmu/\dualL)^i)$. We begin with the following descent lemma.

\begin{lemma}[Descent lemma]
\label{lemma:descent}
Let $f,k : \R^d \to \R \cup \{\infty\}$ satisfy Assumption \ref{assumption:descent}. If $x_0 \in \interior(\dom f)$, then for all $i > 0$, the iterates $x_i$ of Algorithm \ref{alg:dualspaceprecon} are such that $x_i \in \interior(\dom f)$ and for all $x \in \interior(\dom f)$,
\begin{equation}
    \label{eq:descentlemma} k(\grad f(x_i)) \leq k(\grad f(x)) - D_k(\grad f(x), \grad f(x_{i-1})) +  \dualL D_{f}(x_{i-1}, x) -  \dualL D_{f}(x_{i}, x).
\end{equation}
\end{lemma}
\begin{proof}
    Let $C = \interior(\dom f)$. We proceed by induction. For $i = 0$ we have $x_0 \in C$ by assumption. Now, for $i > 0$, assume the induction hypothesis for $x_{i-1}$. Define,
    \begin{equation}
        x_{\lambda} = x_{i-1} - \frac{1}{\lambda} \grad k(\grad f(x_{i-1}))
    \end{equation}
    for $\lambda > 0$. Because $x_{i-1} \in \interior(\dom f) \neq \emptyset$, the set $S = \{\lambda \geq \dualL : x_{\lambda} \in C\}$ is not empty.
    Let $x \in C$. Let $\dualvarx = \grad f(x)$, $\dualvarx_{i-1} = \grad f(x_{i-1})$ and $\dualvarx_{\lambda} = \grad f(x_{\lambda})$ for $\lambda \in S$.
The following identities follow by our definition of $x_{\lambda}$ and some algebra.
            \begin{align}
            \label{eq:statscal}
                \inner{\grad k(\dualvarx_{i-1})}{\dualvarx - \dualvarx_{\lambda}} &= \lambda\inner{x_{i-1} -  x_{\lambda}}{\dualvarx - \dualvarx_{\lambda}},\\
            \label{eq:threepointk}
            \inner{x_{i-1} -  x_{\lambda}}{\grad f(x) - \grad f(x_{\lambda})} &= D_{f}(x_{\lambda}, x) + D_{f}(x_{i-1}, x_{\lambda})- D_{f}(x_{i-1}, x).
            \end{align}
Combining \eqref{eq:statscal} and \eqref{eq:threepointk}, we get
            \begin{equation}
            \label{eq:kdescentid}
            \begin{aligned}
                \lambda D_{f}(x_{i-1}, x_{\lambda})& + \inner{\grad k(\dualvarx_{i-1})}{\dualvarx_{\lambda} - \dualvarx_{i-1}}=\\
                &\lambda D_{f}(x_{i-1}, x)  - \lambda D_{f}(x_{\lambda}, x) + \inner{\grad k(\dualvarx_{i-1})}{\dualvarx - \dualvarx_{i-1}}
            \end{aligned}
            \end{equation}
Putting everything together, we have
    \begin{align}
        k(\dualvarx_{\lambda}) &= k(\dualvarx_{i-1}) + \inner{\grad k(\dualvarx_{i-1})}{\dualvarx_{\lambda} - \dualvarx_{i-1}} + D_{k}(\dualvarx_{\lambda},\dualvarx_{i-1}) \nonumber\\
        &\overset{(a)}{\leq} k(\dualvarx_{i-1}) + \inner{\grad k(\dualvarx_{i-1})}{\dualvarx_{\lambda} - \dualvarx_{i-1}} + \dualL D_{f}(x_{i-1}, x_{\lambda}) \nonumber \\
        &\overset{(b)}{\leq} k(\dualvarx_{i-1}) + \inner{\grad k(\dualvarx_{i-1})}{\dualvarx_{\lambda} - \dualvarx_{i-1}} + \lambda D_{f}(x_{i-1}, x_{\lambda}) \nonumber \\
        &\overset{(c)}{=} k(\dualvarx_{i-1}) + \inner{\grad k(\dualvarx_{i-1})}{\dualvarx-\dualvarx_{i-1}} + \lambda D_{f}(x_{i-1}, x) - \lambda D_{f}(x_{\lambda}, x) \nonumber\\
        \label{eq:descentineq} &\overset{(d)}{\leq} k(\dualvarx) - D_k(\dualvarx, \dualvarx_{i-1}) +  \lambda D_{f}(x_{i-1}, x) - \lambda D_{f}(x_{\lambda}, x).
    \end{align}
    $(a)$ follows from $\dualL$-smoothness, $(b)$ from $\dualL \leq \lambda$ and the non-negativity of the Bregman divergence, $(c)$ from \eqref{eq:kdescentid}, and $(d)$ by definition and simple algebra. Taking $x = x_{i-1}$ and recalling the definition of $\dualvarx_{i-1}$ and $\dualvarx_{\lambda}$ reveals that
    \begin{equation}
        \label{eq:directdescent} k(\grad f(x_{\lambda})) + \lambda D_{f}(x_{\lambda}, x_{i-1}) \leq k(\grad f(x_{i-1})).
    \end{equation}

Now, our goal is to show that $x_{i} = x_{\dualL} \in \interior(\dom f)$ by showing that $\dualL \in S$. We proceed by contradiction, so suppose $\dualL \notin S$. Then $x_{\dualL} \in \R^d \setminus \interior(\dom f)$. Hence we can find $\Lambda \geq \dualL$ such that $x_{\Lambda} \in \partial (\dom f)$. Now take a sequence $\lambda_j \to \Lambda$ such that $\lambda_j > \Lambda$. By the above discussion for all $j \geq 0$ we have $k(\grad f(x_{\lambda_j})) \le k(\grad f(x_{i-1}))$. $k$ being minimized at $0$ means it satisfied Lemma \ref{lemma:radiallyunbounded} and thus is radially unbounded. This implies that $\norm{\grad f(x_{\lambda_j})} \leq c$ for some $c > 0$ and all $j \geq 0$. But this contradicts the requirement that $\norm{\grad f(x_{\lambda_j})} \to \infty$ since $x_{\lambda_j} \to x_{\Lambda} \in \partial(\dom f)$ from the assumption of essential smoothness. This completes the proof that $x_i = x_{\dualL} \in \interior(\dom f)$. Since $\dualL \in S$, \eqref{eq:descentineq} ensures that \eqref{eq:descentlemma} holds.
\end{proof}

We can now provide convergence rates for our method.

\begin{theorem}
\label{thm:convergence}
Let $f,k : \R^d \to \R \cup \{\infty\}$ satisfy Assumption \ref{assumption:descent}. If $x_0 \in \interior(\dom f)$, then for all $i > 0$ the iterates of Algorithm \ref{alg:dualspaceprecon} satisfy
\begin{equation}
    \label{eq:dualgradientrate} k(\grad f(x_i)) - k(0) \leq \frac{\dualL}{i} (f(x_0) - f(\xmin)).
\end{equation}
 In particular, $\grad f(x_i) \to 0$.
 If additionally $f$ is Legendre convex and there exists $\dualmu > 0$ such that $\dualmu D_f(x,y) \leq D_k(\grad f(y), \grad f(x))$ for all $x,y \in \interior(\dom f)$, then for all $i > 0$ the iterates of Algorithm \ref{alg:dualspaceprecon} satisfy
 \begin{equation}
    \label{eq:dualstrngcvxgradientrate} f(x_i) - f(\xmin) \leq \left(1 - \frac{\dualmu}{\dualL}\right)^i (f(x_0) - f(\xmin)).
\end{equation}
 \end{theorem}

\begin{proof}[Proof of Theorem \ref{thm:convergence}]
Let $C = \interior(\dom f)$. We have $x_i \in C$ and $k(\grad f(x_i)) \leq k(\grad f(x_{i-1}))$ by the Descent Lemma \ref{lemma:descent}. We also have $\xmin \in C$ by Lemma \ref{lemma:minimaoflegendrefns}. Finally, \eqref{eq:descentlemma} of Lemma \ref{lemma:descent} with $x = \xmin$ gives us,
\begin{equation}
    \label{eq:descentf}
    \begin{aligned}
        k(\grad f(x_i)) - k(0) &\leq \dualL (f(x_{i-1}) - f(x_i)) - D_k(0, \grad f(x_{i-1}))\\
        &\leq \dualL (f(x_{i-1}) - f(x_i))
    \end{aligned}
\end{equation}
Putting this together, we get
\begin{equation}
    \begin{aligned}
    i (k(\grad f(x_i)) - k(0)) &\leq \sum_{j=1}^i k(\grad f(x_i)) - k(0) \\
    &\leq \dualL (f(x_{0}) - f(x_i)).
    \end{aligned}
\end{equation}
Dividing by $i$ gives our first result. This implies that $k(\grad f(x_i)) \to k(0)$, which implies that $\grad f(x_i) \to 0$ by continuity and the uniqueness of $k$'s minimum. Now, assume that $f$ is Legendre convex and there exists $\dualmu > 0$ such that $\dualmu D_f(x,y) \leq D_k(\grad f(y), \grad f(x))$ for all $x,y \in \interior(\dom f)$. For all $i > 0$,
\begin{equation}
    \begin{aligned}
     \dualL(f(x_i) - f(\xmin)) &\overset{(a)}{\le} \dualL (f(x_{i-1}) - f(\xmin)) - D_k(0, \grad f(x_{i-1}))\\
     &\overset{(b)}{\le} \dualL (f(x_{i-1}) - f(\xmin)) - \dualmu(f(x_{i-1}) - f(\xmin)),
     \end{aligned}
\end{equation}
where $(a)$ follows from \eqref{eq:descentf} and the non-negativity of $k(\dualvarx) - k(0)$. $(b)$ follows from dual relative strong convexity. This inequality implies our desired result.
\end{proof}

Theorem \ref{thm:convergence} guarantees the convergence of the iterates of Algorithm \ref{alg:dualspaceprecon} under the assumption that dual relative smoothness hold globally for a fixed $\dualL$. Unfortunately it may be difficult to to derive a tight bound on $\dualL$ or small $\dualL$ may be appropriate locally. In this case, it may be useful to use a line search to choose $\dualL$. Consider the following generalization of the update rule of Algorithm \ref{alg:dualspaceprecon},
\begin{equation}
    \label{eq:linearsearchdualprecon} x_{i+1} = x_i - \frac{1}{\dualL_i} \grad k (\grad f(x_i))
\end{equation}
where $\dualL_i > 0$ is allowed to depend on the iteration. The next proposition shows that, under suitable assumptions, \eqref{eq:linearsearchdualprecon} converges with rates analogous to Theorem \ref{thm:convergence}.

\begin{proposition}[Adaptive step sizes]
\label{prop:adaptiveconvergence}
Let $f : \R^d \to \R \cup \{\infty\}$ be a proper closed convex function that is differentiable on $\interior(\dom f) \neq \emptyset$ and minimized at $\xmin$. Let $k : \R^d \to \R \cup \{\infty\}$ be a proper closed convex function that is differentiable on $\grad f(\interior(\dom f))$. If $x_0 \in \interior(\dom f)$ and for all $i > 0$ the iterates $x_i$ in \eqref{eq:linearsearchdualprecon} satisfy
\begin{enumerate}
    \item $x_{i}\in \interior(\dom f)$,
    \item \label{enum:descentkLi} $k(\grad f(x_{i}))\le k(\grad f(x_{i-1}))$,
    \item \label{enum:suffdescentLi} $k(\grad f(x_{i})) - k(0)\le \dualL_{i-1} (f(x_{i-1}) - f(x_i))$,
\end{enumerate}
then we have
\begin{equation}
    k(\grad f(x_i)) - k(0)\le \frac{\max_{0 \le j\le i-1} \dualL_j}{i}(f(x_0)-f(\xmin)).
\end{equation}
\end{proposition}

\begin{remark}
In practice, a possible choice of step sizes is
\begin{equation}\label{eq:Lkichoice}
    \dualL_{i-1} =\min\{2^r, r\in \mathbb{Z}: \text{ 1., 2., and 3. of Proposition \ref{prop:adaptiveconvergence} are satisfied}\}.
\end{equation}
If $\dualL$ is the smallest real number such that $f$ is dual $\dualL$-smooth relative to $k$ (see Lemma \ref{lem:equivalentcond} for an equivalent condition), then this scheme satisfies that $\dualL_{i-1} < 2 \dualL$ for every $i > 0$ (hence we are making steps that are almost as large or larger as if we would use the smallest possible fixed $\dualL$, without knowing the value of $\dualL$ in advance). The search through the set in \eqref{eq:Lkichoice} for finding $\dualL_{i}$ can be initialized at $\dualL_{i-1}$.
\end{remark}

\begin{proof}[Proof of Proposition \ref{prop:adaptiveconvergence}]
The proof follows similar lines as in the previous case. First, by summing up the inequalities from \ref{enum:suffdescentLi}, we obtain that
\[\sum_{1\le j\le i}[k(\grad f(x_j)) - k(0)] \le \sum_{1\le j\le i} \dualL_{i-1} ( f(x_{i-1})-f(x_{i}))\le (f(x_0)-f(\xmin)) \max_{0\le j\le i-1} \dualL_j,\]
and using \ref{enum:descentkLi}., it follows that $\sum_{1\le j\le i} [k(\grad f(x_j)) - k(0)] \ge i (k(\grad f(x_i)) - k(0))$. The result follows directly.
\end{proof}

An important question that we do not address in this section is whether the sub-linear convergence of $k(\grad f(x_i)) - k(0)$ implies specific rates of convergence of other quantities of interest. These might be, for example, $\norm{x_i - \xmin}$ or $f(x_i) - f(\xmin)$. Rates for these will likely depend on both $f$ and $k$.

 \section{Applications}
\label{sec:applications}

\subsection{Exponential Penalty Functions}
Consider the following problem.
\begin{equation}\tag{LP}
\label{eq:lpprob}
    \min_{x \in \R^d} \{c^T x : Ax \leq b\},
\end{equation}
where $c \in \R^d$, $b \in \R^n$, and $A \in \R^{n \times d}$. Associate with this linear program the following relaxation into an unconstrained problem:     $\min_{x \in \R^d} f_{\tau}(x)$ for
\begin{align}
\label{eq:lpprobrelax}
    f_{\tau}(x) = c^Tx + \tau \sum_{i=1}^n \exp((A_ix - b_i)/\tau),
\end{align}
where $\tau > 0$ and $A_i$ is the $i$th row of $A$ (a row vector). This approximation of \eqref{eq:lpprob} with exponential penalty functions was studied by several authors (see  \cite{tseng1993convergence, cominetti1994asymptotic, polyak1997nonlinear, auslender1997asymptotic}) and is directly useful in the machine learning literature for boosting (see, \eg{} \cite{meir2003introduction}).
Here we design a dual reference function for $f_\tau$ under the following assumptions.
\begin{assumption}
\label{ass:exppenalty} Suppose that the following hold for problem \eqref{eq:lpprob}.
\begin{enumerate}
    \item $\norm{A_i} = 1$ for $1\le i\le n$.
    \item $A \in \R^{n \times d}$ is of full rank $d \leq n$.
    \item $P = \{x\in R^n: Ax\le b\}$ is a polytope, which is contained in a Euclidean ball of radius $R > 0$ and contains a Euclidean ball of radius $r > 0$.
\end{enumerate}
\end{assumption}
The dual reference function will be designed so that is smooth relative to $f_\tau^*$ and Algorithm \ref{alg:dualspaceprecon}, with appropriate step-size choices, converges with global guarantees.

Define the dual reference function $k : \R^d \to \R$,
\begin{equation}\label{eq:kdef}
k(\dualvarx) = \norm{\dualvarx}-\log(\norm{\dualvarx}+1).\end{equation}
This behaves like a quadratic $\|\dualvarx\|^2/2$ near its minimum $\dualvarx=0$ and like $\|\dualvarx\|$, \ie{} grows linearly, at infinity. It is also possible to verify that $k$ is Legendre convex. Furthermore, we have:
\begin{equation}
    \grad k(\dualvarx) = \frac{\dualvarx}{\norm{\dualvarx}+1}, \qquad \hess k(\dualvarx) =\frac{I}{\|\dualvarx\|+1}-\frac{\dualvarx {\dualvarx}^T}{(\|\dualvarx\|+1)^2 \|\dualvarx\|}.
\end{equation}
Hence, $[\hess k(\dualvarx)]^{-1} \succeq (1+\norm{\dualvarx})I$. From Proposition \ref{lemma:secondorderdualrelativecond} and this inequality it follows that the fact that $k$ is $\dualL$-smooth relative to $f^*_{\tau}$ is implied by
\begin{equation}
\label{eq:equivalentcondnewk}
\grad^2 f_{\tau}(x) \preceq \dualL \l[1+\|\grad f_{\tau}(x)\|\r] I \qquad \forall x\in \R^d.
\end{equation}
This is the strategy of the following theorem, which shows that $f_\tau$ is dual smooth to this choice of $k$ under our assumptions.
\begin{proposition}
\label{thm:exppenaltysmooth}
Under Assumption \ref{ass:exppenalty} for $f_{\tau}$ defined in \eqref{eq:lpprobrelax} and $k$ defined in \eqref{eq:kdef}, we have that
\begin{equation}
\label{eq:exppenaltyclaim}
    \hess f_{\tau}(x) \preceq \dualL_{\tau} [\hess k(\grad f_{\tau}(x))]^{-1} \qquad \forall x \in \R^d,
\end{equation}
where the dual relative smoothness constant is given by
\begin{equation}
    \label{eq:Lrdef}L_{\tau}^{*} = \frac{2R}{r} \frac{\l\|A^TA\r\|}{\tau}   (\eta +\|c\|).
\end{equation}
Here, $\norm{A^TA}$ is the induced matrix norm, and
\begin{equation}
\eta  =\sup_{\norm{s}_{\infty} \le 1} \norm{A^T s} \le \sqrt{n} \norm{A^T}_{\infty}.
\end{equation}
Because $f_\tau$ and $k$ are Legendre convex, $k$ is smooth relative to $f_{\tau}^*$ and Theorem \ref{thm:convergence} implies that  Algorithm \ref{alg:dualspaceprecon} converges with $k(\grad f(x_i))$ converging at a rate $\mathcal{O}(1/i)$.
\end{proposition}
\begin{remark}
From Theorem \ref{thm:convergence}, we have
\begin{equation}
\K(\grad f_{\tau}(x))\le \frac{\dualL_{\tau} (f_{\tau}(x_0)-f_{\tau}(\xmin))}{i}.
\end{equation}
This suggests that, if we can start from an initial point within the polytope, then we can reach a point where $\|\grad f_{\tau}(x)\|$ is significantly less than $\|c\|$ (which is expected to be near the minimum) in polynomial amount of steps, depending on the conditioning $R/r$ and the value of $\tau$. The step-size $1/\dualL_i$ can also be chosen adaptively, as explained in Proposition \ref{prop:adaptiveconvergence}. Near the minimum, both $f_{\tau}(x)$ and $k(\dualvarx)$ behave like quadratic functions, so local linear convergence rates hold. We believe that this iterative scheme is reasonably efficient for high dimensional well-conditioned polytopes, but in other less well conditioned instances it is outperformed by existing algorithms such as multiplicative weights \cite{arora2012multiplicative} or \cite{cominetti1994stable}, which is based on Newton's method (hence uses second-order information).
\end{remark}
\begin{proof}[Proof of Proposition \ref{thm:exppenaltysmooth}]
Note that $1 \leq \eta \leq n$, because $\norm{A_i} = 1$. Let $\alpha(x):=\max_{i\in [n]}(A_ix-b_i)$. Then $\alpha(x)<0$ inside the polytope and $\alpha(x)>0$ outside of it. By differentiation, we have
\begin{align}
    \grad f_{\tau}(x) &= \sum_{i=1}^n A_i \exp((A_ix - b_i)/\tau) + c,\\
    \hess f_{\tau}(x) &= \sum_{i=1}^n \frac{A_i^T A_i}{\tau} \exp((A_ix - b_i)/\tau). \label{eq:hessftau}
\end{align}
Note that $f_\tau$ is defined everywhere and differentiable. Furthermore, under our assumption that $\rank(A^TA) = \rank(A) = d$, it is evidently strictly convex and therefore Legendre.

The Hessian of $f_\tau$ satisfies
    \begin{align}\label{eq:hesstaualphaxbnd}
    \hess f_{\tau}(x) \preceq \exp(\alpha(x)/\tau) \frac{A^TA}{\tau} \preceq  \exp\l(\alpha(x)/\tau\r) \frac{\l\|A^TA\r\|}{\tau} I.
    \end{align}
Because $\eta \geq 1$, it is clear that the claim of the theorem holds for every $x$ where $\alpha(x)\le  0$ (i.e. inside the polytope or on its boundary).
From now on we will assume that $x$ is such that $\alpha(x)> 0$ (outside of the polytope). Let $x_c$ be a minimizer of $\alpha(x)$ (at least one exists since the polytope is compact and $\alpha(x)$ is a continuous function), then using the assumption $\|A_i\|=1$ it follows that $\alpha(x_c)=-r<0$. Hence $x\ne x_c$. We are going to need an upper bound on $\norm{x-x_c}$, which we will obtain as follows. By the definitions, we have $A_{i} x_c\le -r+b_i$ and $A_i x=A_i x-b_i+b_i\le \alpha(x)+b_i$, hence
\begin{align*}
    A_i \l(x_c+\frac{r}{\alpha(x) + r} (x-x_c) \r)&=
    \frac{r}{\alpha(x) + r} A_i x +\frac{\alpha(x)}{\alpha(x) + r}A_i x_c\\
    &\le \frac{r}{\alpha(x) + r}(\alpha(x)+b_i)+
    \frac{\alpha(x)}{\alpha(x) + r}(-r+b_i)=b_i.
\end{align*}
Therefore $x_c+\frac{r}{\alpha(x)+r} (x-x_c) \in P\subset B_{x_c}(2R)$, so
\begin{equation}\label{eq:recnormx}
0<\norm{x-x_c}\le 2\frac{\alpha(x)+r}{r} R \quad \text{ and }\quad  \norm{x-x_c}^{-1} \ge \frac{r}{\alpha(x)+r}  \frac{1}{2 R}.
\end{equation}
Let $\mathcal{I} = \{i \in [n]; \, A_i x - b_i > 0\}$, $\mathcal{J} = \{i \in [n]; \, A_i x - b_i \le 0\}$, and
\begin{equation}
G_{\mathcal{I}}(x) = \sum_{i \in \mathcal{I}} e^{\frac{1}{r}(A_i x - b_i)} A_i \qquad \qquad G_{\mathcal{J}}(x) = \sum_{i \in \mathcal{J}} e^{\frac{1}{r}(A_i x - b_i)} A_i.
\end{equation}
Then $\grad f_{\tau}(x) = G_{\mathcal{I}}(x) + G_{\mathcal{J}}(x)+c$. We have
\begin{align*}
\|G_{\mathcal{I}}(x)\| \ge  \frac{G_{\mathcal{I}}(x)^T(x-x_c)}{\norm{x-x_c}} &= \norm{x-x_c}^{-1}\sum_{i \in \mathcal{I}} e^{\frac{1}{r}(A_i x - b_i)} A_i (x-x_c) \\
&\overset{(a)}{\ge} \norm{x-x_c}^{-1} e^{\frac{\alpha(x)}{\tau}} (\alpha(x)+r) \\
&\overset{(b)}{\ge}  \frac{r}{2R} e^{\frac{\alpha(x)}{\tau}}.
\end{align*}
Here, $(a)$ follows from the facts that there is a $j \in \mathcal{I}$ such that $A_j (x - x_c) = \alpha(x) + b_j - A_j x_c \geq \alpha(x) + r$ and the fact that $A_i (x - x_c) \ge b_i + r - b_i > 0$ holds for every $i \in \mathcal{I}$. $(b)$ follows from \eqref{eq:recnormx}. From \eqref{eq:hesstaualphaxbnd} we obtain that
\begin{equation}
\begin{aligned}
    \hess f_{\tau}(x) &\preceq  \exp\l(\alpha(x)/\tau\r) \frac{\l\|A^TA\r\|}{\tau} I\\
    &\preceq   \frac{2R}{r} \frac{\l\|A^TA\r\|}{\tau} \norm{G_{\mathcal{I}}(x)} I\\
    &\preceq   \frac{2R}{r} \frac{\l\|A^TA\r\|}{\tau} (\norm{\grad f_{\tau}(x)} + \norm{G_{\mathcal{J}}(x)} + \norm{c}) I.
\end{aligned}
\end{equation}
Hence \eqref{eq:equivalentcondnewk} follows from the facts that $\|G_{\mathcal{J}}(x)\| \le \eta$ and $\eta + \norm{c} \geq 1$. As discussed \eqref{eq:exppenaltyclaim} follows from $[\hess k(\dualvarx)]^{-1} \succeq (1+\norm{\dualvarx})I$.
\end{proof}

 \subsection{\texorpdfstring{$p$}{p}-norm Regression}

Consider the following $p$-norm regression problem,
\begin{equation}\tag{pnorm}
    \label{eq:pnormreg}
    \min_{x \in \R^d} \norm{Ax - b}_p^p,
\end{equation}
where $A \in \R^{n \times d}$, $d \ll n$, $b \in \R^n$, and $p \geq 1$. This problem is a useful abstraction for some important graph problems, including Lipschitz learning on graphs \cite{kyng2015algorithms} and $\ell_p$-norm minimizing flows \cite{adil2019faster}. Algorithms specialized for $p$-norm regression have recently been studied in the theoretical computer science literature by several authors (see, \eg{} \cite{bubeck2018homotopy, adil2019iterative} and references therein). In this subsection, we design an appropriate dual reference function for \eqref{eq:pnormreg} under the following assumptions. Let $A_i$ denote the rows of $A$ (as row vectors).
\begin{assumption}
\label{ass:Lp}Suppose that the following hold for problem \eqref{eq:pnormreg}.
\begin{enumerate}
    \item $2 \le p < \infty$.
    \item $A$ is full rank $d$, and for all $x \in \R^d$ there is a subset $I(x) \subset [n]$ such that $A_ix \neq b_i$ for all $i \in I(x)$, and $\mathrm{span} \{A_i: i\in I(x)\}=\R^d$.
    \item
    $c_G=\inf_{\|s\|=1} \norm{As}_p^p>0$.
    \item
     $c_H=\inf_{u,v\in \R^d:\|u\|=1,\|v\|=1}\sum_{i=1}^n \abs{A_i u}^{p-2} (A_i v)^2>0$.
\end{enumerate}
\end{assumption}
\begin{remark} Although these assumptions seem restrictive, we can show that, if $n \ge 2d-1$ and $(A_i)_{1\le i\le n}$ and $(b_i)_{1\le i\le n}$ are chosen as independent random variables with densities that are absolutely continuous with respect to the Lebesgue measure on $\R^{d}$ and $\R$, then the assumptions hold with probability 1. Assumption 2 is implied by the stronger assumption that any $d$ rows of $A$ define a full rank $d$ matrix, and the maximal number of equalities $A_ix=b_i$ that hold for any $x$ is no more than $d$. This stronger version of Assumption 2, and Assumption 3 holds with probability 1 under the random allocation due to the fact that the set of real valued $d\times d$ matrices with determinant $0$ has Lebesgue-measure $0$ in $\R^{d\times d}$ (due to the fact that the determinant is a multivariate polynomial of the entries, and the zero set of such polynomials has Lebesgue measure zero unless they are constant 0, see \cite{caron2005zero}). The minimum in Assumption 4 is achieved for some $u_{\min}$ and $v_{\min}$ due to continuity and compactness of the unit sphere. Since any $d$ rows of $A$ form an independent basis with probability 1, it follows that $u$ and $v$ can be orthogonal to at most $d-1$ of them, respectively, so using $n \ge 2d-1$ there exists an $i$ in the sum $\sum_{i=1}^{n} \abs{A_i u_{\min}}^{p-2} (A_i v_{\min})^2$ that is non-zero, hence Assumption 4 holds.
\end{remark}

Consider the dual reference function $k : \R^d \to \R$,
\begin{equation}\label{eq:kdefpnorm}
k(\dualvarx) = \tfrac{1}{q}\l(\norm{\dualvarx}^2 + 1\r)^{\tfrac{q}{2}} - \tfrac{1}{q},
\end{equation}
for $q=\frac{p}{p-1}$ (hence $\tfrac{1}{p}+\tfrac{1}{q}=1$). This behaves like a quadratic $\| \dualvarx \|^2/2$ near its minimum $\dualvarx=0$ and like $\| \dualvarx \|^q/q$ at infinity. For this $k$, we have
\begin{equation}
\label{eq:kderivspnorm}
\begin{aligned}
    \grad k(\dualvarx) &= \dualvarx (1 + \|\dualvarx\|^2)^{\frac{q-2}{2}}\\
\end{aligned}
\end{equation}
As the next theorem shows dual relative strong convexity and smoothness of $k$ relative to the conjugate of \eqref{eq:pnormreg} hold under our assumptions.

\begin{proposition}\label{thm:Lplinconv}
Let $f(x) = \norm{Ax - b}_p^p$ be the $p$-norm objective. Under Assumption \ref{ass:Lp} for $k$ defined in \eqref{eq:kdefpnorm}, there exists $\dualmu, \dualL > 0$ such that
\begin{equation}
\label{eq:conddualmudualL}
\dualmu [\hess k(\grad f(x))]^{-1} \preceq\hess f(x) \preceq \dualL [\hess k(\grad f(x))]^{-1} \qquad \forall x \in \R^d.
\end{equation}
See  \eqref{eq:dualmu} and \eqref{eq:dualL} for the definitions of $\dualmu$ and $\dualL$. Because $f$ and $k$ are Legendre convex, $k$ is smooth and strongly convex relative to $f^*$ and Theorem \ref{thm:convergence} implies that  Algorithm \ref{alg:dualspaceprecon} converges with $f(x_i) - f(\xmin)$ converging at a linear rate $\mathcal{O}((1-\dualmu/\dualL)^i)$.
\end{proposition}
To test the empirical performance of this method, we have implemented it with $A_i$, $b$, and $x_0$ i.i.d. as standard normals for power $p=4$, $d \in \{10^2, 10^3, 10^4\}$, and $n=10 d$. The inverse step-size $\dualL_0$ was chosen to be $\dualL_0=1$ initially, and multiplied by 2 if the function value would increase due to too large steps (hence this was chosen adaptively in the beginning, but $\dualL_i$ was never decreased later on). As Figure \ref{fig:Lp} shows, empirically our method seems to be performing well, with high precision achieved after 50-80 gradient evaluations, and the convergence rate seems to be mostly unaffected by the dimension $d$. Hence in this random setting dual space preconditioning is indeed very efficient, and competitive with previous works \cite{bubeck2018homotopy, adil2019iterative, adil2019faster} which had dimension dependent convergence rates. We think that based on Proposition \ref{thm:Lplinconv}, it can be shown that with high probability, dimension-free convergence rates hold in this random scenario when the number of vectors $n$ tends to infinity (the proof would be based on concentration inequalities for empirical processes, see e.g. \cite{boucheron2013concentration} for an overview of such inequalities).
Note however that we do not believe this always to be the case for general non-random $A$ and $b$, and there could be instances of very poor conditioning (such as when $n\approx d$) where the homotopy method of \cite{bubeck2018homotopy} or the IRLS method of \cite{adil2019fast} could perform  better.
\begin{figure}
    \centering
    \includegraphics[scale=0.4]{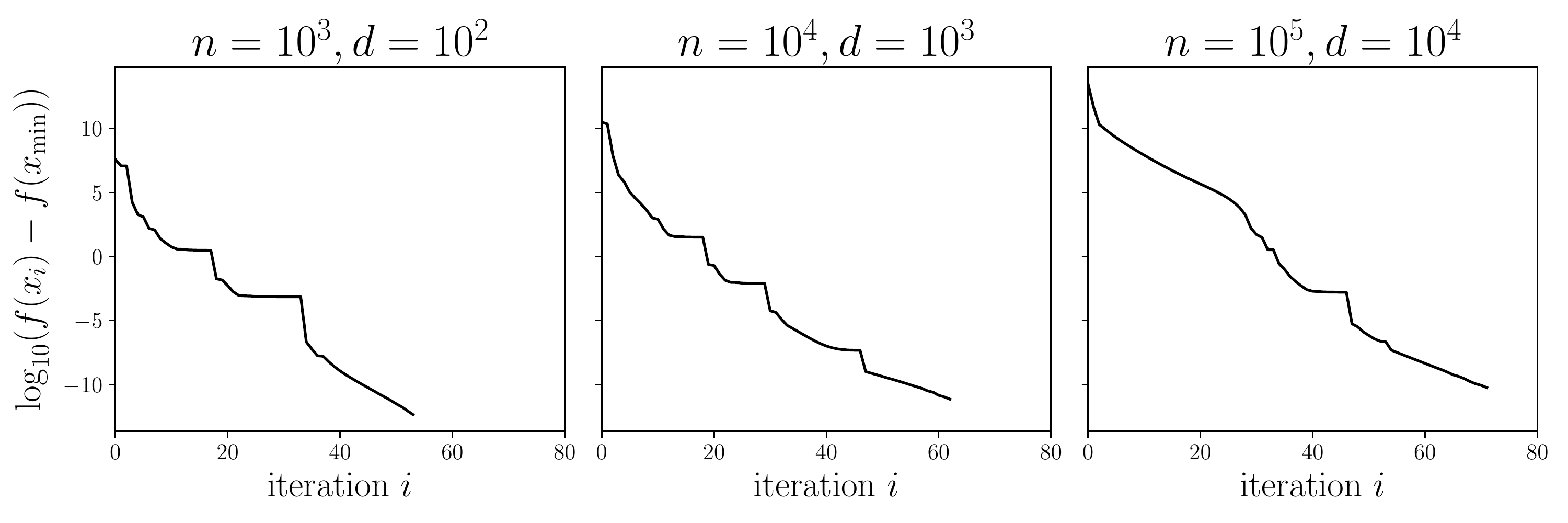}
    \caption{Convergence rates for $p$-norm regression are mostly unaffected by the dimension $d$ for these random instances with $p=4$.}
    \label{fig:Lp}
\end{figure}
The proof of Proposition \ref{thm:Lplinconv} is based on the following two lemmas.
\begin{lemma}[Bounds on the gradient]\label{lem:gradbnd}
Let $f(x) = \norm{Ax - b}_p^p$ be the $p$-norm objective for \eqref{eq:pnormreg}. Under Assumption \ref{ass:Lp}, we have
\begin{equation}
    L_G \|x\|^{p-1}-C_{G}\le \|\grad f(x)\|\le U_G \|x\|^{p-1}+D_{G}
\end{equation}
for all $x \in \R^d$, with constants
\begin{align*}
    L_G&=2^{-p+1} c_G=2^{-p+1} \inf_{\|s\|=1} \norm{As}_p^p, \quad C_G=\l(\sum_{i=1}^{n} |b_i|^{p}\r)^{(p-1)/p} \cdot c_G^{1/p},\\
    U_G&=2^{p-2}(p+1) \sup_{\|s\|=1} \norm{As}_p^p,\quad D_G=2^{p-2}(p-1)
    \l(\sum_{i=1}^{n} |b_i|^{p}\r).
\end{align*}
\end{lemma}
\begin{proof}
By differentiation, we have
\begin{equation}
    \label{eq:Lpgradf}
    \grad f(x) = p \sum_{i=1}^n \abs{A_i x - b_i}^{p-2}(A_i x - b_i) A_{i},
\end{equation}
thus
\begin{align*}
    &\|\grad f(x)\| = p \l\|\sum_{i=1}^n \abs{A_i x - b_i}^{p-2}(A_i x - b_i) A_{i}\r\|\\
    &\ge \max\l(\frac{p}{\|x\|}\sum_{i=1}^n \abs{A_i x - b_i}^{p-2}(A_i x - b_i) A_{i}x,0\r)
    \\
    &= \max\l(\frac{p}{\|x\|}\sum_{i=1}^n \l[\abs{A_i x - b_i}^{p-2}(A_i x - b_i)^2 + \abs{A_i x - b_i}^{p-2}(A_i x - b_i)b_i\r],0\r)
    \\
    &\ge \max\l(\frac{p}{\|x\|}\sum_{i=1}^n \l(\abs{A_i x - b_i}^{p}-\abs{A_i x - b_i}^{p-1}|b_i|\r),0\r),
    \intertext{now by Young's inequality $\abs{A_i x - b_i}^{p-1}|b_i|\le \abs{A_i x - b_i}^{p} \frac{p-1}{p}+\frac{\abs{b_i}^p}{p}$, hence}
    &\ge \max\l(\frac{1}{\|x\|}\sum_{i=1}^n \l(\abs{A_i x - b_i}^{p}-|b_i|^p\r),0\r)\\
    \intertext{ using the fact that $\abs{a+b}^p\le (\abs{a}+\abs{b})^p=\l(\frac{2\abs{a}+2\abs{b}}{2}\r)^p\le 2^{p-1} (\abs{a}^p+\abs{b}^p)$ by convexity (this is so-called the $C_p$ inequality),
    so $\abs{A_i x - b_i}^p+\abs{b_i}^p\ge 2^{-p+1}  \abs{A_i x}^p$, hence}
    &\ge \max\l(\frac{1}{\|x\|}\sum_{i=1}^n \l(2^{-p+1}\abs{A_i x}^{p}-2|b_i|^p\r),0\r)\\
    &\ge \max\l(2^{-p+1} \l[\inf_{\|s\|=1}\|As\|_p^p\r] \cdot \|x\|^{p-1}-\frac{2\sum_{i=1}^{n}|b_i|^p}{\|x\|},0\r),
\end{align*}
and the lower bound follows from Assumption \ref{ass:Lp} by straightforward rearrangement. For the upper bound, notice that
\begin{align*}
\|\grad f(x)\| &\le p \sup_{\|v\|=1}\sum_{i=1}^n \abs{A_i x - b_i}^{p-1} |A_{i}v|\\
&\le 2^{p-2}p \sup_{\|v\|=1}\sum_{i=1}^n \l(\abs{A_i x}^{p-1}|A_{i}v|+\abs{b_i}^{p-1}|A_{i}v|\r)\\
&\le 2^{p-2}p \l[\|x\|^{p-1}\sup_{\|s\|=1,\|v\|=1}\sum_{i=1}^n \l(\abs{A_i s}^{p-1}|A_{i}v|\r)+\sup_{\|v\|=1}\sum_{i=1}^n\abs{b_i}^{p-1}|A_{i}v|\r]\\
&\le 2^{p-2}p \l[\frac{p+1}{p}\sup_{\|s\|=1}\|As\|_p^p+\frac{p-1}{p}\sum_{i=1}^n\abs{b_i}^{p}\r],
\end{align*}
hence the result follows. The last step uses Fenchel-Young, and rearrangement.
\end{proof}

\begin{lemma}[Bounds on the Hessian]\label{lem:Hessbnd}
Let $f(x) = \norm{Ax - b}_p^p$ be the $p$-norm objective. Suppose that Assumption \ref{ass:Lp} holds, and let
\begin{align}\label{eq:RHdef}R_H&=\l\|\sum_{i=1}^{n}|b_i|^{p-2}A_{i}^T A_{i}\r\|^{1/(p-2)}/(c_H 2^{-p})^{1/(p-2)},\\
\label{eq:rhoHdef}
\rho_H&=\inf_{\|x\|\le R_H} \lambda_{\min}(\grad^2 f(x))=
\inf_{\|x\|\le 1, \|u\|=1}p (p-1) \sum_{i=1}^n \abs{A_i x - b_i}^{p-2} (A_i u)^2.
\end{align}
Then $\rho_H>0$, and we have
\begin{equation}\label{eq:lemHessbnd}
(L_H \|x\|^{p-2}+C_{H})I\preceq \grad^2 f(x)\preceq (U_H \|x\|^{p-2}+D_{H})I
\end{equation}
for all $x \in \R^d$, with constants
\begin{align*}
    L_H&=\min\l(p (p-1)2^{-p-1} c_H, \frac{\rho_{H}}{2R_H^{p-2}}\r),\\
    U_H&=2^{p-3}p(p-1)\sup_{\|u\|=1,\|v\|=1}\sum_{i=1}^n \abs{A_i u}^{p-2} (A_i v)^2,\\
    C_H&=\min\l(\frac{\rho_H}{2},p (p-1)2^{-p-1} c_H R_H^{p-2}\r),
    D_H=p (p-1) 2^{p-3}\l\|\sum_{i=1}^n |b_i|^{p-2} A_{i}^T A_{i}\r\|.
\end{align*}
\end{lemma}
\begin{proof}
We have by differentiation
\begin{equation}
\label{eq:LpHessf}
    \hess f(x) = p (p-1) \sum_{i=1}^n \abs{A_i x - b_i}^{p-2} A_{i}^T A_{i}.
\end{equation}
Notice that using the fact that $|a-b|^{p-2}+|b|^{p-2}\ge 2^{-(p-1)}|a|^{p-2}$, we have
\begin{align*}
    \hess f(x) &= p (p-1) \sum_{i=1}^n \abs{A_i x - b_i}^{p-2} A_{i}^T A_{i}\\
    &\succeq p (p-1) \sum_{i=1}^n \l(2^{-(p-1)}|A_ix|^{p-2}- |b_i|^{p-2}\r) A_{i}^T A_{i}\\
    &\succeq p (p-1)2^{-(p-1)} c_H \|x\|^{p-2}-p (p-1)\l\|\sum_{i=1}^{n}|b_i|^{p-2}A_{i}^T A_{i}\r\|.
\end{align*}
Let $R_H$ be as in \eqref{eq:RHdef}, then using the above bound, we can see that for $\|x\|\ge R_H$, we have
\begin{equation}\label{eq:outsideofradiusRH}
\begin{aligned}
    \hess f(x) &\succeq p (p-1)2^{-p} c_H \|x\|^{p-2} I \\
    &\succeq p (p-1)2^{-p-1} c_H \|x\|^{p-2}+p (p-1)2^{-p-1} c_H R_H^{p-2}.
\end{aligned}
\end{equation}
Since the minimum of the continuous function $\lambda_{\min}(\hess f(x))$ is achieved on the compact set $B_{R_H}$, and by the second part of Assumption \ref{ass:Lp}, it cannot be zero, and hence $\rho_{H}>0$ and
$\grad^2 f(x)\succeq \rho_H I$ for every $x\in B_{R_H}$. The lower bound in \eqref{eq:lemHessbnd} follows by combining this with \eqref{eq:outsideofradiusRH}. For the upper bound, using the inequality $|a+b|^{p-2}\le 2^{p-3}(|a|^{p-2}+|b|^{p-2})$, we obtain that
\begin{align*}
    \hess f(x) &\preceq p (p-1) 2^{p-3}\sup_{\|s\|=1}
    \l\|\sum_{i=1}^n \abs{A_i s}^{p-2} A_{i}^T A_{i}\r\|\cdot \|x\|^{p-2}\\
    &+p (p-1) 2^{p-3}\l\|\sum_{i=1}^n |b_i|^{p-2} A_{i}^T A_{i}\r\|.
\end{align*}
\end{proof}

Now we are ready to prove our main result in this section.
\begin{proof}[Proof of Proposition \ref{thm:Lplinconv}]
First, both $f$ and $k$ are Legendre convex in this case. This is easy to verify for $k$, and evidently $f$ is differentiable everywhere. To verify strict convexity of $f$, note that $\hess f(x) \succ 0$ under part two of Assumption \ref{ass:Lp}. Since both $f$ and $\K$ are twice differentiable, by Proposition \ref{lemma:secondorderdualrelativecond}, it suffices to check that \eqref{eq:conddualmudualL} holds for the linear convergence of Algorithm \ref{alg:dualspaceprecon}. We have by differentiation,
\begin{equation}
        \hess k(\dualvarx) = (1 + \|\dualvarx\|^2)^{\frac{q-2}{2}} I + (q-2) (1 + \|\dualvarx\|^2)^{\frac{q-4}{2}}  \dualvarx{\dualvarx}^T.
\end{equation}
Now it is easy to see that for $p \in [2, \infty)$, we have $q = p/(p-1) \in (1, 2]$ and it is not difficult to verify that $\grad^2 \K$ satisfies that for all $x^* \in \R^d$,
\begin{equation}
    (1 + \|x^*\|^2)^{\frac{1}{2} \frac{p-2}{p-1}} I \preceq \l[\grad^2 \K(x^*) \r]^{-1} \preceq (p-1) (1 + \|x^*\|^2)^{\frac{1}{2} \frac{p-2}{p-1}} I.
\end{equation}
The claim of the theorem now follows by some straightforward rearrangement using Lemmas \ref{lem:gradbnd} and \ref{lem:Hessbnd}, with constants
\begin{align}\label{eq:dualmu}
    \dualmu&=\min\l(\frac{C_H}{2(p-1)(2+2D_G)},\frac{L_H}{4(p-1)U_G^{(p-2)/(p-1)}}\r),\\
    \label{eq:dualL}
    \dualL&=\min\l(\frac{U_H}{(L_G/2)^{(p-2)/(p-1)}},
    4U_H \l(\frac{C_G}{L_G}\r)^{(p-2)/(p-1)}+2D_H\r).
\end{align}
\end{proof}
  \section{Discussion}
\label{sec:discussion}
In this paper we introduced a non-linear preconditioning scheme for gradient descent on Legendre convex functions $f$ that converges under generalizations of the standard Lipschitz assumption on $\grad f$. There are at least two interpretations of this method. The first is as a generalization of gradient descent in which the update direction is preconditioned by the gradient map $\grad k$ of a designed dual reference, Legendre convex function $k$. The second interpretation is as a Bregman gradient method in the dual space, which minimizes the designed $k$ while the conjugate $f^*$ plays the role of the ``reference function''. The choice of $k$ affects the conditioning of our method, which is made explicit in our analysis through a relative smoothness condition between $k$ and $f^*$. The dual relative conditions admit non-smooth $f$ and $k$, and are provably distinct dual cousins of the relative smoothness conditions introduced by \cite{bauschke2016descent}. $k$ serves as a model of the convex conjugates $f^*$ in a certain problem class. In section \ref{sec:applications}, we show how this method can be applied to exponential penalty functions  (see, \eg{} \cite{cominetti1994asymptotic, cominetti1994stable}) and $p$-norm regression (see \cite{bubeck2018homotopy,adil2019iterative} and references therein) with global convergence rate guarantees.

Algorithm \ref{alg:dualspaceprecon} is related to a number of existing methods, some of which are subject to the analysis we provide. The most notable of these is the method of steepest descent with respect to a given norm $\norm{\cdot}$ (now not necessarily Euclidean). Here we follow the exposition of Boyd and Vandenberghe \cite[sect. 4.9]{boyd2004convex}. The steepest descent iteration is given by
\begin{equation}
\label{eq:steepestdescent}
\begin{aligned}
    x_{i+1} = x_i + \frac{1}{L} \norm{\nabla f(x_i)}_* d, \qquad \text{where } d \in \underset{\norm{x} \leq 1}{\arg \max} \inner{-\nabla f(x_i)}{x},
\end{aligned}
\end{equation}
and $\norm{\dualvarx}_* = \sup_{\norm{x} \leq 1} \inner{x}{\dualvarx}$ is the dual norm of $\norm{\cdot}$. The identity $\partial (\norm{\dualvarx}_*^2/2) = \norm{\dualvarx}_*\arg \max\{\inner{\dualvarx}{x} : \norm{x} \leq 1\}$  for all $\dualvarx \in \R^d$ implies that
for strictly convex and differentiable $\norm{\cdot}_*$, the steepest descent method \eqref{eq:steepestdescent} is a special case of dual preconditioned gradient descent with $k(\dualvarx) = \norm{\dualvarx}_*^2/2$. Our analysis does not apply in the case of other norms or normalized steepest descent \cite{boyd2004convex}. Algorithm \ref{alg:dualspaceprecon} also generalizes the rescaled gradient method of \cite[sect. 2.2]{wilson2019accelerating}.
Thus, our method may be seen as a generalization of the steepest descent method and rescalings of gradient descent. Dual preconditioning is more distantly related to the dual gradient methods \cite{tseng1991applications, beck2014fast}. These methods are designed for problems with non-smooth, but strongly convex structure. They exploit the duality between classical smoothness and strong convexity by applying smooth minimization algorithms to a dual problem.
Similarly, Algorithm \ref{alg:dualspaceprecon} can be seen as a move to the dual space, in which a dual problem $k(\dualvarx) \approx f^*(\dualvarx) - \inner{\dualvarx}{\xmin}$ (dual to $f(x) + \delta_{x = \xmin}(x)$) is minimized by a Bregman gradient method. Thus, dual gradient methods and dual preconditioning are most easily applied when the dual structure is relatively more benign to model than the primal structure, \eg{} when $f$ has super-quadratic growth.

There are a couple natural questions that arise from this work. First, it may be useful to pursue the analogy with dual gradient methods further and to design methods for the general composite model that exploit dual relative smoothness. Second, there is still considerable difficulty in the design of $k$. Thus, it may be productive to investigate whether methods from linear preconditioning  (see \cite{benzi2002preconditioning} for a review), such as incomplete factorizations or sparse approximate inverses, can be generalized to the non-linear setting for the design of $k$. Nonetheless, the dual relative conditions studied in this work provide new avenues for improving the conditioning of optimizers via hard-won domain-specific knowledge.

 \section*{Acknowledgements}
We thank the anonymous referees for their insightful comments that helped us to improve the paper. We thank David Balduzzi for insightful comments, Patrick Rebeschini for suggesting exponential penalty functions and Sushant Sachdeva for suggesting $p$-norm regression. CJM acknowledges the support of the Institute for Advanced Study and the James D. Wolfensohn Fund, the support of a DeepMind Graduate Scholarship, and the support of the Natural Sciences and Engineering Research Council of Canada under reference number PGSD3-460176-2014. YWT's research leading to these results has received funding from the European Research Council under the European Union's Seventh Framework Programme (FP7/2007-2013) ERC grant agreement no. 617071. This material is based upon work supported in part by the U.S. Army Research Laboratory and the U. S. Army Research Office, and by the U.K. Ministry of Defence (MoD) and the U.K. Engineering and Physical Research Council (EPSRC) under grant number EP/R013616/1.

\bibliographystyle{plainnat}
\bibliography{refs}

\end{document}